\documentclass[11pt,a4paper]{article}
\usepackage{amssymb}
\usepackage{stackengine,amsmath}
\stackMath
\usepackage{latexsym}
\usepackage{amsthm}

\newcommand{\V}{\textup{Vol}(4_1)}

\newtheorem{thm}{Theorem}
\newtheorem{lem}[thm]{Lemma}
\newtheorem{cor}[thm]{Corollary}
\newtheorem{prop}[thm]{Proposition}
\theoremstyle{definition}
\newtheorem*{remark}{Remark}

\usepackage{graphicx}
\usepackage{url}
\usepackage{xpatch}
\xpatchcmd{\proof}{\itshape}{\normalfont\proofnameformat}{}{}
\newcommand{\proofnameformat}{}

\DeclareMathOperator{\sgn}{sgn}

\usepackage{xcolor}
\usepackage{caption}
\usepackage{varwidth}
\newcommand{\captionbackgroundcolor}[1]{\colorlet{cpbgcol}{#1}}
\DeclareCaptionFont{black}{\color{black}}
\captionbackgroundcolor{black!3}
\newcommand\blfootnote[1]{%
  \begingroup
  \renewcommand\thefootnote{}\footnote{#1}%
  \addtocounter{footnote}{-1}%
  \endgroup
}

\begin{document}

\renewcommand{\proofnameformat}{\bfseries}

\begin{center}
{\Large\textbf{A conjecture of Zagier and the value distribution of quantum modular forms}}

\blfootnote{\noindent \textbf{Keywords}: quantum modular forms, quantum knot invariants, colored Jones polynomial, Kashaev invariant, Sudler products, continued fractions. \textbf{Mathematics Subject Classification (2020)}: 57K16, 11J70, 11L03, 26D05, 60F05.}

\vspace{5mm}

\textbf{Christoph Aistleitner and Bence Borda}

\vspace{5mm}

{\footnotesize Graz University of Technology

Institute of Analysis and Number Theory

Steyrergasse 30, 8010 Graz, Austria

Email: \texttt{aistleitner@math.tugraz.at} and \texttt{borda@math.tugraz.at}}
\end{center}

\begin{abstract}
In his influential paper on quantum modular forms, Zagier developed a conjectural framework describing the behavior of certain quantum knot invariants under the action of the modular group on their arguments. More precisely, when $J_{K,0}$ denotes the colored Jones polynomial of a knot $K$, Zagier's modularity conjecture describes the asymptotics of the quotient $J_{K,0} (e^{2 \pi i \gamma(x)}) / J_{K,0}(e^{2 \pi i x})$ as $x \to \infty$ along rationals with bounded denominators, where $\gamma \in \mathrm{SL}(2,\mathbb{Z})$. This problem is most accessible for the figure-eight knot $4_1$, where the colored Jones polynomial has a simple explicit expression in terms of the $q$-Pochhammer symbol. Zagier also conjectured that the function $h(x) = \log (J_{4_1,0} (e^{2 \pi i x}) / J_{4_1,0}(e^{2 \pi i /x}))$ can be extended to a function on $\mathbb{R}$ which is continuous at irrationals. In the present paper, we prove Zagier's continuity conjecture for all irrationals for which the sequence of partial quotients in the continued fraction expansion is unbounded. In particular, the continuity conjecture holds almost everywhere on the real line. We also establish a smooth approximation of $h$, uniform over all rationals, in accordance with the modularity conjecture. As an application, we find the limit distribution (after a suitable centering and rescaling) of $\log J_{4_1,0}(e^{2 \pi i x})$, when $x$ ranges over all reduced rationals in $(0,1)$ with denominator at most $N$, as $N \to \infty$, thereby confirming a conjecture of Bettin and Drappeau.
\end{abstract}

\section{Introduction}

The colored Jones polynomials $J_{K,n}, n \geq 2$, and the Kashaev invariants $\langle K \rangle _N,~ N \geq 2$, are two quantum knot invariants that have been intensively studied in the mathematical and theoretical physics literature. The two invariants are related via $\langle K \rangle_N = J_{K,N} (e^{2 \pi i/N})$. Among their most interesting features are the connection with quantum field theory, and the link to the hyperbolic geometry of knot complements via the volume conjecture. For more basic background, see for example \cite{DG,GA,MY,WI}.

The colored Jones polynomial $J_{K,n}$ can be defined as a certain Laurent polynomial arising from Skein relations of the knot (for $n \geq 2$), and by periodicity it can be extra\-polated backward to give $J_{K,n}$ also for $n \leq 0$. We do not give a more detailed general definition, since in the present paper we will only be concerned with $J_{4_1,0}$, which has the explicit representation
\begin{equation} \label{j41}
J_{4_1,0} (q) = \sum_{n=0}^\infty \left|(1-q) (1-q^2) \cdots (1-q^n)\right|^2,
\end{equation}
defined for all complex roots of unity $q$. Note that the series in this definition is actually finite at roots of unity. In the formulas above ``$4_1$'' is the Alexander--Briggs notation for the figure-eight knot, which is the only knot with crossing number four, and in many regards the simplest hyperbolic knot. Using the $q$-Pochhammer symbol $(q;q)_n$ one can write
$$
J_{4_1,0} (q) = \sum_{n=0}^\infty |(q;q)_n|^2,
$$
which hints at a connection with enumerative combinatorics and (mock) modular forms (cf.\ for example \cite{BB}). We note in passing that for other hyperbolic knots $J_{K,0}$ has a somewhat similar, but more complicated representation in terms of $q$-Pochhammer symbols.

The volume conjecture concerns the exponential growth rate of the Kashaev invariant, and asserts that
$$
\lim_{N \to \infty}  \frac{2 \pi \log |\langle K \rangle_N|}{N} = \textup{Vol} (K), 
$$
where $\textup{Vol} (K)$ denotes the hyperbolic volume of the complement of $K$. Writing $\mathbf{J}_K (x) = J_{K,0} (e^{2 \pi i x})$ throughout this paper, the volume conjecture can also be written in terms of $\mathbf{J}_K$ since $\langle K \rangle_N = \mathbf{J}_K (1/N)$. The volume conjecture has been generalized to the arithmeticity conjecture, which predicts the full series expansion of $|\mathbf{J}_K (1/N)|$ as $N \to \infty$. Both conjectures have been proved for knots with a small number of crossings, including the figure-eight knot, for which
\begin{equation} \label{volume_4_1}
\textup{Vol} (4_1) = 4 \pi \int_0^{5/6} \log (2 \sin (\pi x)) \, \mathrm{d}x \approx 2.0299,
\end{equation}
but they remain open for general knots $K$. 

In his seminal paper on quantum modular forms \cite{ZA}, Zagier mentions the colored Jones polynomial as the ``most mysterious and in many ways the most interesting'' among the examples listed in the paper (even if it, strictly speaking, does not satisfy his definition of a quantum modular form). For fixed $N$, the numbers $\mathbf{J}_K(a/N),~1 \leq a \leq N,~\gcd(a,N)=1,$ are the Galois conjugates of $\langle K \rangle_N$ in $\mathbb{Q} (e^{2 \pi i/N})$. Zagier's modularity conjecture is a vast generalization of the volume and arithmeticity conjectures, addressing the Galois invariant spreading of the Kashaev invariant on the set of complex roots of unity. The modularity conjecture has also been proved for all hyperbolic knots with a small number of crossings, including the figure-eight knot \cite{BD1}, but remains open for general $K$. We do not replicate the statement of the modularity conjecture in full generality here, and refer to \cite{ZA} instead. Roughly speaking, the modularity conjecture makes a detailed prediction about the asymptotic expansion of the quotient $\mathbf{J}_{K} (\gamma(x)) / \mathbf{J}_{K} (x)$ as $x \to \infty$ along rationals $x$ with bounded denominators, where $\gamma \in \mathrm{SL}(2,\mathbb{Z})$. In the special case when $x$ runs along the positive integers and $\gamma(x) = -1/x$, this leads back to the arithmeticity conjecture.

In the present paper we will be concerned with the case of $K = 4_1$ and $\gamma(x) = -1/x$, whence the quantity considered in Zagier's modularity conjecture (after taking a logarithm and switching a sign) becomes
\begin{equation} \label{h_def}
h(x) = \log \frac{\mathbf{J}_{4_1}(x)}{\mathbf{J}_{4_1}(1/x)} , \qquad x \in \mathbb{Q} \backslash \{ 0 \}.
\end{equation}
This function $h$ will be the central object of study in the present paper. Note that $\mathbf{J}_{4_1}(x)$ is easily seen to be invariant under the map $x \mapsto x+1$, and that the two maps $\gamma(x) = x+1$ and $\gamma(x) = -1/x$ together generate the full modular group. Note also that  $h(x)=-h(1/x)$ and $h(x)=h(-x)$, so it is sufficient to understand $h$ on the interval $(0,1)$. Zagier's paper \cite{ZA} contains several plots of $h$, including one similar to our Figure \ref{fig:h} below (the ``global'' plot), as well as plots showing the behavior of $h$ near rationals with small denominators and near badly approximable irrationals, similar to our Figures \ref{fig:psi_1_10} and \ref{fig:psi_sqrt}. He observes that the ``global'' plot misleadingly suggests that $h$ could be monotonically decreasing, which is not actually the case; compare the comment after Figure \ref{fig:psi} below. Furthermore, he writes that the experimental evidence is 
\begin{quote}
 ``[\dots] seeming to indicate that the function $h(x)$ is continuous but in general not differentiable at irrational values of $x$.''
\end{quote}
Since $h(x)$ is only defined over the rationals, the continuity at irrationals is of course to be understood with respect to the real topology; in other words, Zagier suggests that $h(x)$ can be extended to a function on $\mathbb{R}$ (rather than $\mathbb{Q}$) that is continuous at irrationals. In the present paper we prove that indeed this is the case, with the possible exclusion of irrationals which are badly approximable in the sense of Diophantine approximation. Recall that the badly approximable numbers are exactly those which have bounded partial quotients in their continued fraction expansions. (Some background on Diophantine approximation is given at the beginning of Section \ref{trig_section}). 

\begin{figure}[ht!]  
\begin{center}
\includegraphics[width=0.75 \linewidth]{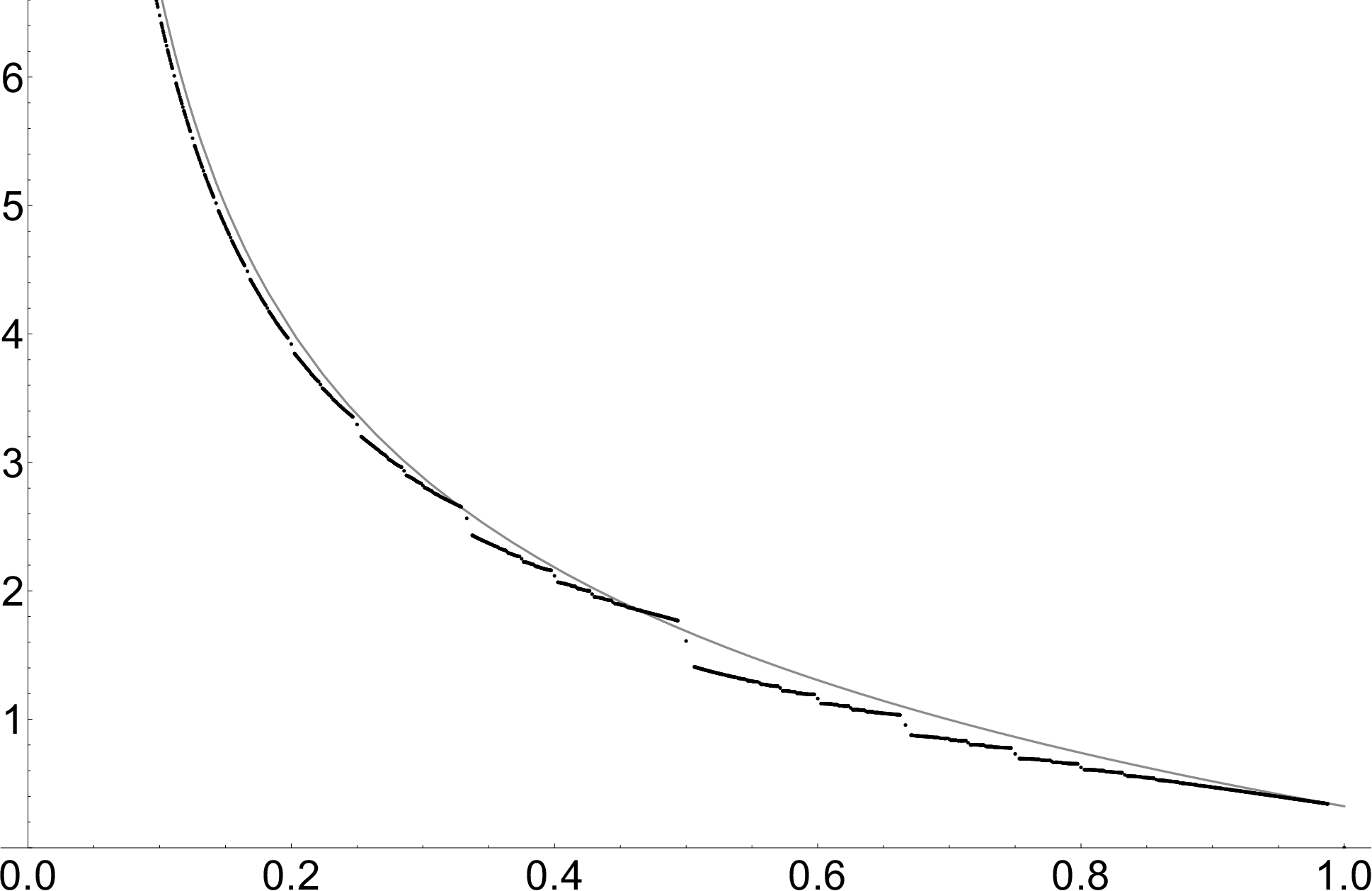}
\end{center}
  \caption{The function $h(x)$, evaluated at all rationals in $(0,1)$ with denominator at most $80$ (black graph with jumps). For comparison, the plot also shows the function $\frac{\textup{Vol}(4_1)}{2 \pi x} - \frac{3}{2} \log x$ (gray solid line), which is suggested as a continuous approximation to $h(x)$ by Formula \eqref{formula_h}.}  \label{fig:h}
\end{figure}

\begin{thm} \label{th1}
Assume that $\alpha$ is an irrational whose sequence of partial quotients in the continued fraction expansion is unbounded. Then $\lim_{x \to \alpha} h(x)$ exists and is finite.
\end{thm}

It is well known that the set of badly approximable numbers has vanishing Lebesgue measure. Thus Theorem \ref{th1} implies that the answer to Zagier's continuity problem is positive almost everywhere.

\begin{cor} \label{cor_th1}
The function $h(x)$ can be extended to an almost everywhere continuous function on $\mathbb{R}$.
\end{cor}

Our more technical Theorem \ref{quantitativecontinuitytheorem}, which will be stated in Section \ref{cont_section}, gives quantitative upper bounds for the maximal fluctuation of $h$ on intervals defined in terms of a joint initial segment of the continued fraction expansion. This can be read as establishing a modulus of continuity for $h$ at $\alpha$, which takes into account the size of the partial quotients in the continued fraction expansion of $\alpha$.

Theorem \ref{th1} leaves the continuity of $h$ at badly approximable irrationals open. It will be seen that our argument crucially relies on the existence of an unbounded subsequence of partial quotients, so some essential new ideas will be necessary to treat the case of badly approximable $\alpha$. Some partial results for quadratic irrational $\alpha$ (when the sequence of partial quotients is eventually periodic) are contained in our earlier paper \cite{AB2}. In this case Zagier's continuity problem might be more tractable than in the general case, due to the additional structure coming from the periodicity of the continued fraction expansion. The case of general badly approximable $\alpha$ (with no particular structure in the sequence of partial quotients) seems to be even more challenging.

Zagier's modularity conjecture in the special case of $h(x)$ implies in particular that
\begin{equation} \label{formula_h}
h(x) = \frac{\textup{Vol}(4_1)}{2 \pi x} - \frac{3 \log x}{2} - \frac{\log 3}{4} + o(1)
\end{equation}
as $x \to 0^+$ along rationals with bounded numerators; this is in accordance with the numerical data (cf.\ Figure \ref{fig:h}), which in fact seems to suggest that the same holds as $x \to 0^+$ along all rationals. Defining
\begin{equation} \label{psi_def}
\psi(x) := h(x) - \frac{\textup{Vol}(4_1)}{2 \pi x} + \frac{3 \log x}{2}, \qquad x \in (0,1) \cap \mathbb{Q},
\end{equation}
the function $\psi$ seems to capture very well the local ``irregular'' aspects of $h$, such as the jumps at rationals, and certain self-similarity properties (see Figure \ref{fig:psi}). It seems that so far hardly anything was known about the (maximal) size of $\psi$. For example, while the plots clearly indicate that Zagier's function $h(x)$ should be bounded on $(0,1)$ as long as one stays away from $x=0$, we believe that nothing of that sort was known so far. Here we will prove the following bound. 

\begin{thm} \label{th3} \label{hasymptoticstheorem} 
We have
\[ h(x) = \frac{\V}{2 \pi x} + O \left( 1+ |\log x| \right) \]
for all rationals $x \in (0,1)$, with a universal implied constant.
\end{thm}
In particular, $h$ is locally bounded on $(0,1)$. We conjecture that Theorem \ref{th3} can actually be improved to
\begin{equation} \label{psi_conj}
h(x) = \frac{\V}{2 \pi x} - \frac{3 \log x}{2} + O(1),
\end{equation}
or equivalently that $\psi$ is bounded on $(0,1)$, as suggested by Figure \ref{fig:psi}.

\begin{figure}[ht!]  
\begin{center}
  \includegraphics[width=0.75 \linewidth]{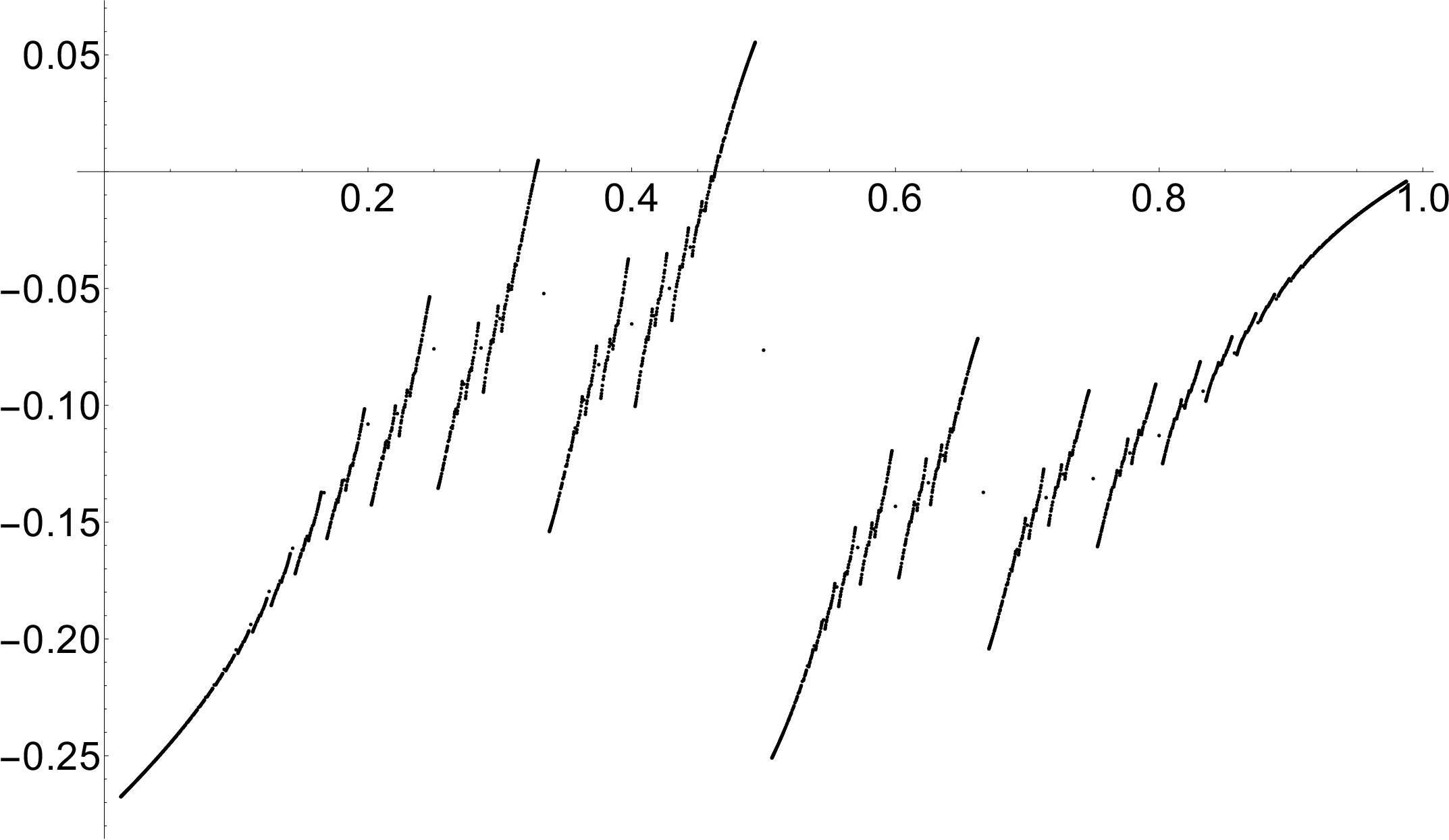}
\end{center}
  \caption{The function $\psi(x) = h(x) - \textup{Vol}(4_1)/(2 \pi x) + (3/2) \log x$, evaluated at all rationals in $(0,1)$ with denominator at most $80$. Note the apparent self-similar structure of $\psi$. Note also the isolated function values at rationals with small denominators such as $x=1/2$ or $x=1/3$, and that $\lim_{x \to 0} \psi(x)$ appears to be $-\frac{\log 3}{4} \approx -0.275$ and $\lim_{x \to 1} \psi(x)$ appears to be 0, in accordance with the arithmeticity and modularity conjectures.} \label{fig:psi}
\end{figure}

\begin{figure}[ht!]  
\begin{center}
  \includegraphics[width=0.75 \linewidth]{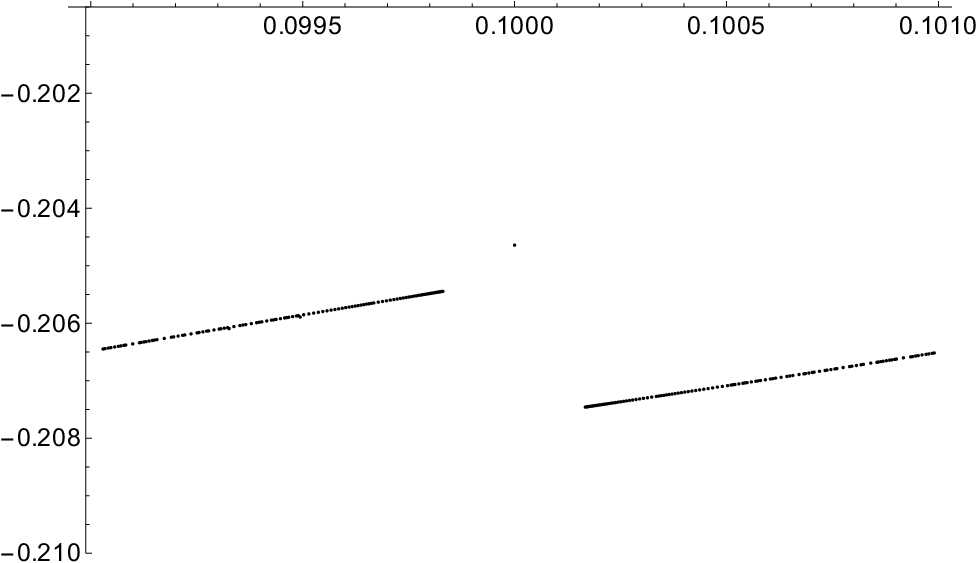}
\end{center}
  \caption{The function $\psi(x)$, evaluated at all rationals with denominator at most $600$ in a small neighborhood of $x=1/10$. Note the isolated function value at $x=1/10$, and the very regular behavior when approaching 1/10 from the left or from the right. Note also that the ``global'' plot in Figure \ref{fig:psi} might seem to indicate that $\psi$ consists of a continuous increasing component which is interceded by a discrete decreasing component, and that the values of $\psi$ at rationals are always situated between the corresponding left and right limits, i.e.\ $\lim_{x \to r^{-}} \psi(x) > \psi(r) > \lim_{x \to r^{+}} \psi(x)$. However, as this figure indicates, this is probably not true for some (small?) rationals, where actually $\lim_{x \to r^{-}} \psi (x) < \psi(r)$, i.e.\ an initial \emph{upward} jump is followed by a downward jump. It might still be the case that $\lim_{x \to r^{-}} \psi(x) > \lim_{x \to r^{+}} \psi(x)$ at all rationals $r \in (0,1)$; at least we have not found a counterexample.} \label{fig:psi_1_10}
\end{figure}

\begin{figure}[ht!]  
\begin{center}
  \includegraphics[width=0.75 \linewidth]{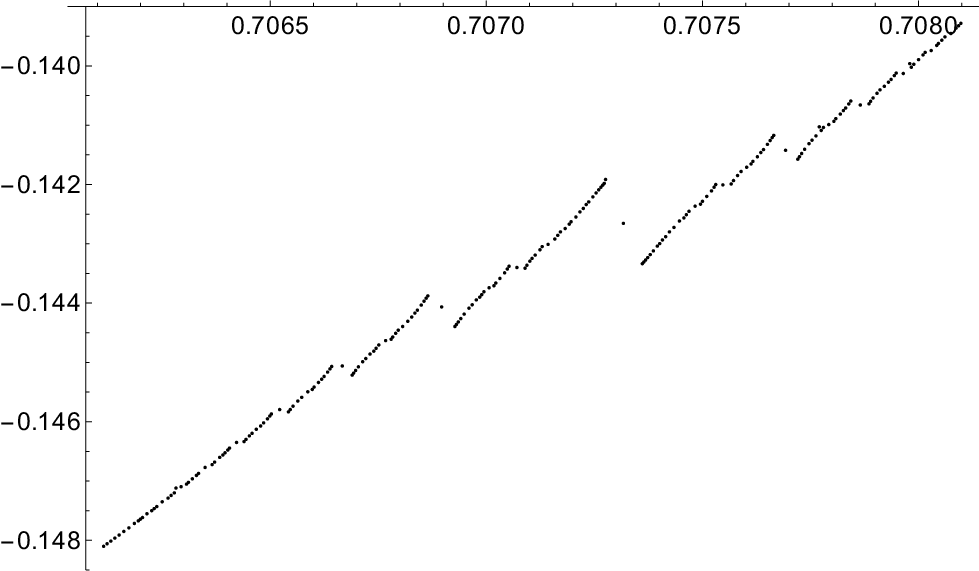}
\end{center}
  \caption{The function $\psi(x)$, evaluated at all rationals with denominator at most $600$ in a small neighborhood of $x=1/\sqrt{2}$. When compared to Figure \ref{fig:psi_1_10} above, one can see the different bevahior of $\psi$ near rationals with small denominators and near badly approximable irrationals (note that the scaling is the same in both plots, making them directly comparable).} \label{fig:psi_sqrt}
\end{figure}

Zagier's continuity problem is interesting in its own right, but as observed in \cite{BD1}, it also has implications on the value distribution of $\log \mathbf{J}_{4_1} (x)$ as $x$ ranges over all rationals with their denominators bounded by a given threshold. More precisely, \cite{BD1} contains a detailed prediction for the limit distribution of a suitably centered and rescaled version of $\log \mathbf{J}_{K} (x)$, and in \cite[Theorem 4]{BD1} it was shown that a positive answer to the conjecture in Equation \eqref{psi_conj} above, together with a positive answer to Zagier's continuity problem, would imply the validity of this prediction for $K=4_1$. It turns out that our Theorems \ref{th1} and \ref{th3}, together with a result of Bettin and Drappeau \cite{BD3} on the distribution of sums of partial quotients of random rationals, are sufficient to establish unconditionally Bettin and Drappeau's conjecture on the value distribution of $\log \mathbf{J}_{4_1}$.

\begin{thm} \label{th2}
Let $F_N$ denote the set of all reduced rationals in $(0,1)$ with denominator at most $N$. There exists a constant $D$ such that for every interval $[a,b] \subset \mathbb{R}$,
\[
\lim_{N \to \infty} \frac{1}{\# F_N}  \# \left\{ x \in F_N: \left( \frac{\log \mathbf{J}_{4_1}(x)}{\frac{\textup{3 Vol}(4_1)}{\pi^2} \log N} - \frac{2}{\pi} \log \log N - D \right) \in [a,b] \right\}  = \int_a^b g(y) \, \mathrm{d}y,
\]
where
$$
g(y) = \frac{1}{2 \pi} \int_{-\infty}^\infty e^{-ity - |t| \left( 1 + \frac{2i}{\pi} \sgn(t) \log |t| \right)} \, \mathrm{d}t
$$
is the density function of the standard stable law with stability parameter $1$ and skewness parameter $1$.
\end{thm}

Many properties of the functions $h$ and $\psi$ remain mysterious. A conjecture stated in \cite{ZA} would imply in particular that the left and right limits $\lim_{x \to r^{-}} h(x)$ and $\lim_{x \to r^{+}} h(x)$ exist at all rationals $r \in (0,1)$ (or equivalently, that these limits exist for $\psi$). This would be in accordance with the impression given by Figure \ref{fig:psi_1_10}, but remains unproved. Supposing that these limits exist, the actual size of the jumps from left to right limit also remains unclear, as well as the relation of the left and right limits to the actual function value $h(r)$. As Figures \ref{fig:h} and \ref{fig:psi} indicate, there seem to be larger jumps at rationals with small denominators, but we are lacking a detailed understanding of this phenomenon. 

Before continuing with the proofs and the underlying heuristics, we describe the structure of the remaining parts of the paper. In Sections \ref{prod_section} and \ref{cot_section} we will introduce certain shifted trigonometric products and shifted cotangent sums, which are the key objects appearing throughout the proofs. The heuristic principles underlying the proof of the continuity problem can only be explained after introducing these auxiliary objects, and we will consequently present the heuristic reasoning in Section \ref{heu_section} below. In particular, we will show why it is necessary to have unbounded partial quotients in the continued fraction expansion of $\alpha$, as a consequence of a natural approximate ``factorization'' of $\mathbf{J}_{4_1}$ arising from the Ostrowski representation of positive integers. This factorization will be made precise in Section \ref{factorsection}. In Section \ref{tailsection} we show that the ``tail'' in this approximate factorization is (surprisingly) well-behaved, which allows us to prove Theorem \ref{th1} in Section \ref{cont_section}. In Section \ref{asy_section} we prove Theorem \ref{th3}, and in Section \ref{41_section} we prove Theorem \ref{th2}. 
\section{Trigonometric products and cotangent sums} \label{trig_section}

Throughout this paper we make heavy use of concepts and results from Diophantine approximation and the theory of continued fractions. Among the standard textbooks on these subjects are the books of Bugeaud \cite{bugeaud}, Cassels \cite{cassels}, Khinchin \cite{khinchin}, Rockett and Sz\"usz \cite{rocks} and Schmidt \cite{schmidt}. Below we state a few basic facts on continued fractions, which can all be found in each of the textbooks mentioned above.

Throughout the paper, $\alpha$ denotes a real number. If $\alpha$ is rational, then it can be represented as a (finite) continued fraction in the form 
\begin{equation} \label{cont_fr}
\alpha =  a_0 + \frac{1\kern6em}{\displaystyle
  a_1 + \frac{1\kern5em}{\displaystyle
    a_2 +\stackunder{}{\ddots\stackunder{}{\displaystyle
      {}+ \frac{1}{\displaystyle
        a_{L-1} + \frac{1}{a_L}}}}}}
\end{equation}
for a suitable number $L$ (called the length of the continued fraction), a suitable integer $a_0$, and suitable positive integers $a_1, \dots, a_L$ (which are called the partial quotients of $\alpha$). For each rational $\alpha$ there is a unique form of such a representation for which $a_L >1$, and throughout the paper it is understood that always this form of the continued fraction expansion is used. We will write all continued fractions in the form $\alpha =[a_0;a_1, a_2, \dots, a_L]$,  which is just a shorthand version of \eqref{cont_fr}. The numbers $p_{\ell}/q_{\ell}=[a_0;a_1,\dots, a_{\ell}],~ 0 \leq \ell \leq L,$ are called the (continued fraction) convergents to $\alpha$.  If $\alpha$ is irrational, then its continued fraction expansion is infinite, and we write $\alpha=[a_0;a_1,a_2,\ldots ]$ with the convention that $L=\infty$ and $q_L=\infty$. Irrationals whose sequence of partial quotients is bounded are called badly approximable; as noted before the statement of Corollary \ref{cor_th1}, the set of badly approximable numbers has vanishing Lebesgue measure (but full Hausdorff dimension).

We denote the distance from the nearest integer function by $\| \cdot \|$, and the fractional part function by $\{ \cdot \}$. Among the fundamental facts from the theory of continued fractions, which will be frequently used in this paper, are the recursive formulas 
\[ q_{\ell +1}=a_{\ell +1}q_{\ell} + q_{\ell -1} \qquad \textrm{and} \qquad \| q_{\ell +1} \alpha \|=-a_{\ell +1} \| q_{\ell} \alpha \| + \| q_{\ell -1} \alpha \|, \]
the identities
\[ q_{\ell} p_{\ell -1} - p_{\ell} q_{\ell -1} =(-1)^{\ell} \qquad \textrm{and} \qquad q_{\ell} \| q_{\ell -1} \alpha \| + q_{\ell -1} \| q_{\ell} \alpha \| =1, \]
and the estimate $1/(a_{\ell +1}+2) \le q_{\ell} \| q_{\ell} \alpha \| \le 1/a_{\ell +1}$. Furthermore, a key basic ingredient in our proofs is the fact that every integer $0 \le N <q_L$ has a unique Ostrowski expansion $N=\sum_{\ell =0}^{L-1} b_{\ell}(N) q_{\ell}$, where $0 \le b_0(N) < a_1$ and $0 \le b_{\ell}(N) \le a_{\ell +1}$ are integers which satisfy the extra rule that $b_{\ell}(N) =0$ whenever $b_{\ell +1}(N)=a_{\ell +2}$. We use the convention $b_{\ell}(N)=0$ for $\ell \ge L$.

Throughout the paper, $I_R$ denotes the indicator of a relation $R$, and $C>0$ a universal constant whose value changes from line to line. By convention, empty sums equal $0$, and empty products equal $1$.

\subsection{Shifted trigonometric products} \label{prod_section}

Throughout the paper, we write $P_N$ for the trigonometric product
\[ P_N (\alpha) := \prod_{n=1}^N |2 \sin (\pi n \alpha)| \qquad (N \in \mathbb{N}) . \]
This is called Sudler product, and has a long history going back at least to a paper of Erd\H{o}s and Szekeres \cite{ESZ}. See \cite{AB1,AB2,ATZ,GKN,GKN2,GN,GNZ,h1,h2,MV} for recent papers concerned with the (asymptotic) behavior of such products. Note that 
\[ P_N (\alpha) = \left| (1 - e^{2 \pi i \alpha}) (1 - e^{2 \pi i 2 \alpha}) \cdots (1 - e^{2 \pi i N \alpha}) \right|,\]
so that in accordance with \eqref{j41}, the value of the colored Jones polynomial at a reduced rational $p/q$ can be written as
\[ \mathbf{J}_{4_1}(p/q) = \sum_{N=0}^{q-1} P_N(p/q)^2 . \]
We will also need a shifted form of the Sudler product, namely
\[ P_N (\alpha, x) := \prod_{n=1}^N |2 \sin (\pi (n \alpha +x))| \qquad (N \in \mathbb{N}, \,\, x \in \mathbb{R}) . \]
One of our main tools is the product form \cite{AB2,GKN}
\begin{equation}\label{pnproductform}
P_N (\alpha ) = \prod_{\ell =0}^{L-1} \prod_{b=0}^{b_{\ell}(N)-1} P_{q_{\ell}} \left( \alpha, (-1)^{\ell} \frac{b q_{\ell} \| q_{\ell} \alpha \| + \varepsilon_{\ell}(N)}{q_{\ell}} \right) ,
\end{equation}
where
\begin{equation}\label{epsilondef}
\varepsilon_{\ell} (N) := q_{\ell} \sum_{m=\ell +1}^{L-1} (-1)^{\ell +m} b_m (N) \| q_m \alpha \| .
\end{equation}
In fact, $\varepsilon_{\ell}(N)$ only plays a role if $b_{\ell}(N) \ge 1$, otherwise it does not appear in the product form \eqref{pnproductform}. The following simple facts were observed in \cite{AB1,AB2}; for the sake of completeness, we include the short proof.
\begin{lem}\label{epsilonlemma} For any $0 \le \ell <L$ such that $b_{\ell} (N) \ge 1$, we have
\[ -q_{\ell} \| q_{\ell} \alpha \| + q_{\ell} \| q_{\ell +1} \alpha \| \le \varepsilon_{\ell} (N) \le q_{\ell} \| q_{\ell +1} \alpha \| . \]
If $b_{\ell +1}(N) \le (1-\delta) a_{\ell +2}$ with some $\delta>0$, then $\varepsilon_{\ell}(N) \ge -(1-\delta /3) q_{\ell} \| q_{\ell} \alpha \|$. If $b_{\ell +2}(N) \le (1-\delta) a_{\ell +3}$ with some $\delta >0$, then $\varepsilon_{\ell}(N) \le (1-\delta /3) q_{\ell} \| q_{\ell +1} \alpha \|$.
\end{lem}

\begin{proof} Keeping only the nonnegative terms in the definition \eqref{epsilondef} of $\varepsilon_{\ell}(N)$, we obtain the upper bound
\[ \begin{split} \varepsilon_{\ell}(N) &\le q_{\ell} \left( b_{\ell +2}(N) \| q_{\ell +2} \alpha \| + b_{\ell +4} (N) \| q_{\ell +4} \alpha \| + \cdots \right) \\ &\le q_{\ell} \left( a_{\ell +3} \| q_{\ell +2} \alpha \| + a_{\ell +5} \| q_{\ell +4} \alpha \| + \cdots \right) \\ &= q_{\ell} \left( (\| q_{\ell +1} \alpha \| - \| q_{\ell +3} \alpha \| ) + ( \| q_{\ell +3} \alpha \| - \| q_{\ell +5} \alpha \| ) +\cdots \right) \\ &= q_{\ell} \| q_{\ell +1} \alpha \| , \end{split} \]
as claimed. Similarly, keeping only the nonpositive terms and using the fact that $b_{\ell}(N) \ge 1$ implies $b_{\ell +1}(N)<a_{\ell +2}$, we obtain the lower bound
\[ \begin{split} \varepsilon_{\ell}(N) &\ge q_{\ell} \left( -b_{\ell +1}(N) \| q_{\ell +1} \alpha \| - b_{\ell +3} (N) \| q_{\ell +3} \alpha \| - \cdots \right) \\ &\ge q_{\ell} \left( -(a_{\ell +2}-1) \| q_{\ell +1} \alpha \| - a_{\ell +4} \| q_{\ell +3} \alpha \| - \cdots \right) \\ &= q_{\ell} \left( \| q_{\ell +1} \alpha \| + (\| q_{\ell +2} \alpha \| - \| q_{\ell} \alpha \| ) + ( \| q_{\ell +4} \alpha \| - \| q_{\ell +2} \alpha \| ) +\cdots \right) \\ &= -q_{\ell} \| q_{\ell} \alpha \| + q_{\ell} \| q_{\ell +1} \alpha \| , \end{split} \]
as claimed. An obvious modification of the argument leads to the last two estimates.
\end{proof}

\subsection{A shifted cotangent sum} \label{cot_section}

A close connection between the Sudler product and certain cotangent sums was first observed by Lubinsky \cite{LU}, and more recently in \cite{AB1,AB2,BD1}. In our proofs we will use a cotangent sum estimate of Lubinsky, or more precisely, a generalization of such an estimate to a shifted version of the same sum. For further related results we refer to \cite{BC2,BD2}. 
\begin{lem}\label{cotangentsumlemma} Let $1 \le \ell \le L$ and $0 \le N <q_{\ell}$. For any real $|x|<\| q_{\ell -1} \alpha \|$,
\[ \left| \sum_{n=1}^N \cot \left( \pi \left( n \alpha + x \right) \right) \right| \ll q_{\ell} \left( \frac{1}{1-\frac{|x|}{\| q_{\ell -1} \alpha \|}} + \log \max_{1 \le m \le \ell} a_m \right) . \]
For any real $|x|<1/q_{\ell}$,
\[ \left| \sum_{n=1}^N \cot \left( \pi \left( \frac{np_{\ell}}{q_{\ell}} + x \right) \right) \right| \ll q_{\ell} \left( \frac{1}{1-q_{\ell}|x|} + \log \max_{1 \le m \le \ell} a_m \right) . \]
The implied constants are universal.
\end{lem}

\begin{proof} Let $F(x):=\sum_{n=1}^N \cot ( \pi ( n \alpha + x ))$ denote the shifted cotangent sum in the first formula. Lubinsky \cite[Theorem 4.1]{LU} proved the estimate
\[ |F(0)| \ll q_{\ell} \left( 1+\log \max_{1 \le m \le \ell} a_m \right) \]
for the unshifted sum. The derivative of $F(x)$ satisfies
\[ |F'(x)| = \left| \sum_{n=1}^N \frac{-\pi}{\sin^2 (\pi (n \alpha +x))} \right| \ll \sum_{n=1}^N \frac{1}{\| n \alpha +x \|^2} . \]
Since $1 \le n \le N <q_{\ell}$ by assumption, the best rational approximation property gives $\| n \alpha \| \ge \| q_{\ell -1} \alpha \|$. Thus by the triangle inequality for $\| \cdot \|$,
\[ \| n \alpha +x \| \ge \| n \alpha \| - |x| \ge \| n \alpha \| \left( 1-\frac{|x|}{\| q_{\ell -1} \alpha \|} \right) , \]
and so
\[ |F'(x)| \ll \frac{1}{\left( 1-\frac{|x|}{\| q_{\ell -1} \alpha \|} \right)^2} \sum_{n=1}^{q_{\ell}-1} \frac{1}{\| n \alpha \|^2} \ll \frac{q_{\ell}^2}{\left( 1-\frac{|x|}{\| q_{\ell -1} \alpha \|} \right)^2} . \]
The last inequality follows from a classical method based on the pigeonhole principle, see e.g.\ \cite{BO} for a detailed proof. Therefore
\[ |F(x)| \le |F(0)| + \left| \int_0^x F'(y) \, \mathrm{d}y \right| \ll q_{\ell} \left( \frac{1}{1-\frac{|x|}{\| q_{\ell -1} \alpha \|}} + \log \max_{1 \le m \le \ell} a_m \right) , \]
as claimed. The second formula of the lemma follows from the first formula applied to a suitable sequence of $\alpha$'s converging to $p_{\ell}/q_{\ell}$, say $\alpha=[a_0;a_1,\dots, a_{\ell},M]$ as $M \to \infty$.
\end{proof}

\subsection{The heuristic picture} \label{heu_section}

Let $r \in (0,1)$ be a rational with continued fraction expansion $r=[0;a_1,a_2,\dots, a_L]$. We are interested in the value of $\mathbf{J}_{4_1} (r)$ and want to relate it to $\mathbf{J}_{4_1} (1/r)$, where it plays a crucial role that the continued fraction expansion of $1/r$ very similar to the one of $r$, namely $1/r = [a_1;a_2,a_3,\dots,a_L]$. By the periodicity of the trigonometric functions involved, we can discard the integer part of $1/r$, and consider $r' := \{ 1/r \} = [0;a_2,a_3,\dots,a_L]$ instead. So we will be concerned with 
$$
h(r)=\log \frac{\mathbf{J}_{4_1} (r)}{\mathbf{J}_{4_1} (r')} ,
$$
and try to figure out how the absence of the first partial quotient $a_1$ in $r'$ affects this expression. In the end, we will let $r \to \alpha$ for some fixed irrational $\alpha$ (which is assumed to have unbounded partial quotients). 

Let $p_\ell/q_\ell,~0 \leq \ell \le L$ denote the convergents to $r$; in particular, $r=p_L/q_L$. As noted in Section \ref{prod_section}, we have
$$
\mathbf{J}_{4_1} (r) = \sum_{N=0}^{q_L-1} P_N(r)^2,
$$
where $P_N(r)$ has the factorization
\begin{equation} \label{factor_heur}
P_N(r) = \prod_{\ell =0}^{L-1} \prod_{b=0}^{b_{\ell}(N)-1} P_{q_{\ell}} \left( r, (-1)^{\ell} \frac{b q_{\ell} \| q_{\ell} r \| + \varepsilon_{\ell}(N)}{q_{\ell}} \right)
\end{equation}
in terms of the Ostrowski representation of $N$. By a rough approximation we have 
\begin{equation} \label{approx1}
q_\ell \|q_\ell r \| \approx 1/a_{\ell+1}.
\end{equation}
A general identity\footnote{We thank Michael Henry (TU Graz) for pointing out to us that this is Kubert's functional equation with parameter 1, written in multiplicative form. See \cite{Kub,Mil} for a proof and further applications.} says that for any reduced rational $p/q$,
$$
|2\sin(\pi x/q)| P_{q-1} (p/q,x/q) = |2\sin (\pi x)| .
$$
This identity suggests that since $r \approx p_{\ell}/q_{\ell}$, we can expect
\begin{equation} \label{approx3}
|2\sin(\pi x/q_{\ell})| P_{q_\ell-1} (r,x/q_{\ell}) \approx |2\sin (\pi x)|.
\end{equation}
Upon identifying  $|2 \sin(\pi x/q_{\ell})| \approx |2 \sin(\pi (q_\ell r + x/q_{\ell}))|$ as essentially being the $q_\ell$-th factor of the shifted Sudler product, we end up with
$$
P_{q_\ell} (r,x/q_{\ell}) \approx |2\sin (\pi x)|
$$
as a rough approximation. Under appropriate circumstances this is not far from the truth; cf.\ Figure 1 and Theorem 5 of \cite{AB1}. When using this heuristics in \eqref{factor_heur}, together with \eqref{approx1} and for the moment ignoring the numbers $\varepsilon_\ell(N)$, we obtain
$$
\log P_N(r)^2 \approx \sum_{\ell=0}^{L-1} \sum_{b=0}^{b_\ell(N)-1} 2 \log |2 \sin (\pi b / a_{\ell+1} )|.
$$
Interpreting the inner sum as a Riemann sum, we thus have
$$
\log P_N(r)^2 \approx \sum_{\ell=0}^{L-1} a_{\ell+1} \Psi(b_\ell(N) /a_{\ell+1}),
$$
with $\Psi(y) = 2 \int_{0}^y \log |2 \sin (\pi x)| \, \mathrm{d}x$. The function $\Psi(y)$ is maximized at $y=5/6$, reflecting the fact that $2\sin (5 \pi /6) = 1$. By \eqref{volume_4_1}, the maximal value is
$$
\Psi(5/6) = \frac{\textup{Vol}(4_1)}{2 \pi},
$$
which explains how $\textup{Vol}(4_1)$ enters into formulas such as \eqref{formula_h}.

The argument above allows us to identify those $N$ for which $P_N(r)$ is particularly large as essentially being those $N$ which have many of their Ostrowski coefficients satisfy $b_\ell \approx (5/6) a_{\ell+1}$ (in particular for those $\ell$ for which $a_{\ell+1}$ is large). Taking a sum over all $N$ we obtain
\begin{equation} \label{approx2}
\mathbf{J}_{4_1} (r) = \sum_{N=0}^{q_L-1} P_N(r)^2 \approx \sum_{(b_0, b_1,\dots,b_{L-1})} \prod_{\ell=0}^{L-1} \exp \left( a_{\ell+1} \Psi (b_\ell /a_{\ell+1}) \right) ,
\end{equation}
where the last sum is meant as a sum over all admissible Ostrowski expansions $(b_0,b_1,\dots,b_{L-1})$ of an integer $N < q_L$. We will factorize this sum over all $(b_0, b_1,\dots,b_{L-1})$ into a sum over only the first $k$ Ostrowski coefficients $(b_0,\dots,b_{k-1})$, multiplied with a sum over the remaining ones $(b_k,b_{k+1}, \dots, b_{L-1})$, and for reasons which will be explained later, we have to do so at a position $k$ such that the partial quotient $a_{k+1}$ is ``large.'' Assuming that the first and second segment of all such potential Ostrowski expansions are ``independent,'' the last expression in \eqref{approx2} should be roughly
\begin{equation}\label{approx5}
\sum_{(b_0, b_1,\dots,b_{k-1})} \prod_{\ell=0}^{k-1} \exp \left( a_{\ell+1} \Psi(b_\ell /a_{\ell+1}) \right) \times \sum_{(b_k, \dots,b_{L-1})} \prod_{\ell=k}^{L-1} \exp \left( a_{\ell+1} \Psi (b_\ell /a_{\ell+1}) \right) .
\end{equation}
Repeating this procedure for $\log P_N(r')^2$ yields a decomposition into $\sum_{(b_1, \dots, b_{k-1})}$ and $\sum_{(b_k,\dots, b_{L-1})}$ similar to the one above, but now with the ``digit'' $b_0$  missing in the first sum since the partial quotient $a_1$ is missing in the continued fraction expansion of $r'$. This extra digit in $\mathbf{J}_{4_1} (r)$ can contribute a factor of size roughly $a_1 \max_{y \in (0,1)} \Psi(y) = a_1 \textup{Vol}(4_1)/(2 \pi)$; since $a_1 \approx 1/r$ for small $r$, this explains why $h(r) \approx \textup{Vol}(4_1)/(2 \pi r)$ in first approximation. When finally considering $\log \mathbf{J}_{4_1}(r) - \log \mathbf{J}_{4_1}(r')$ as $r \to \alpha$ with some irrational $\alpha$, the effect of this extra digit $a_1$ will ``stabilize'' as $k$ increases, and the second part of \eqref{approx5} will asymptotically be the same for $\mathbf{J}_{4_1}(r)$ and $\mathbf{J}_{4_1}(r')$ since it arises from the part of the continued fraction expansion which is the same for $r$ and $r'$, thereby leading to an overall convergent behavior of $\log \mathbf{J}_{4_1}(r) - \log \mathbf{J}_{4_1}(r')$ as $r \to \alpha$.

There are many challenges when trying to implement this approach. First, \eqref{approx3} is not an equality, since $r \neq p_\ell/q_\ell$, but rather $r = p_\ell/ q_\ell + \eta_{\ell}$ for some (small) $\eta_{\ell}$. For the Sudler product we thus have, after taking logarithms,
$$
\log P_{q_\ell} (r) = \sum_{n=1}^{q_\ell} \log |2 \sin(\pi n (p_\ell / q_\ell + \eta_{\ell}))|,
$$
and a similar formula holds for the shifted Sudler products; using the linearization $\log |2 \sin(\pi n (p_\ell / q_\ell + \eta_{\ell}))| \approx \log |2 \sin(\pi n p_\ell / q_\ell)| + \pi n \cot(\pi n p_\ell/q_\ell) \eta_{\ell}$ we are led to the cotangent sums that were introduced in Section \ref{cot_section}. Detailed estimates for such cotangent sums form a key technical ingredient in this paper. 

Another critical problem is that the decomposition from line \eqref{approx2} to \eqref{approx5} is very delicate. Avoiding the coarse approximation of $P_{q_\ell}$ by $\Psi$ which led to \eqref{approx2}, but still using \eqref{approx1}, we try to factorize
\begin{equation}\label{approx6}
\sum_{N=0}^{q_L-1} \prod_{\ell =0}^{L-1} \prod_{b=0}^{b_{\ell}(N)-1} P_{q_{\ell}} \left( r, (-1)^{\ell} \frac{b/a_{\ell+1} + \varepsilon_{\ell}(N)}{q_{\ell}} \right)
\end{equation}
into a sum over all admissible Ostrowski coefficients $(b_0, \dots, b_{k-1})$ times a sum over all admissible $(b_k,\dots,b_{L-1})$. There are two crucial ``dependence'' effects to consider here. On the one hand, the ``digits'' in the Ostrowski expansion are not independent in the appropriate stochastic sense. This is in marked contrast with other numeration systems such as the decimal system, where the digits are stochastically independent. In the case of the Ostrowski numeration system the situation is made more complicated by the extra rule that $b_{\ell} =0$ whenever $b_{\ell +1}=a_{\ell +2}$, consequently the stochastic structure is described by a Markov chain (see \cite{ds} for details). Even though there is no true independence, the degree of dependence between $b_\ell$ and $b_{\ell+1}$ decreases when the maximal possible value of $b_{\ell+1}$ (i.e.\ the number $a_{\ell+2}$) becomes larger; this comes from the fact that the necessity of applying the extra rule becomes less likely. Thus the factorization of the sum in \eqref{approx6} at a certain position $k$ into two sums over $(b_0, \dots, b_{k-1})$ resp.\ $(b_k,\dots,b_{L-1})$ requires that $a_{k+1}$ (i.e.\ the maximal possible value of $b_k$) is large -- this is one place where our assumption on the existence of large partial quotients is of crucial use. 

Secondly, there is another source of dependence, which comes from the fact that the numbers $\varepsilon_\ell$ for $\ell \leq k-1$ (in the first part of the desired factorization) depend also on the values of $b_k,\dots,b_{L-1}$ (in the second part of the factorization); cf.\ the definition of $\varepsilon_\ell$ in \eqref{epsilondef}. This is \emph{not} the same effect as the one described in the previous paragraph, which was only concerned with the inherent dependence of the Ostrowski numeration system -- now we have another source of  dependence which comes from the specific definition of our products $P_N$.

The strategy for the solution is the following. The value of $\mathbf{J}_{4_1}(r)$ is a sum over $q_L$ different products $P_N(r)^2$, but only relatively few of them actually make a contribution of significant size. As indicated above, a significant contribution comes only from those numbers $N$ for which the most relevant Ostrowski coefficients, namely $b_\ell$ for which $a_{\ell+1}$ is large, satisfy $b_\ell \approx (5/6) a_{\ell+1}$. In particular, if we know that $a_{k+1}$ is large ($k$ being the index where we try to split the summation in \eqref{approx6}, as in \eqref{approx5}), then we can show that there is a significant contribution to $\mathbf{J}_{4_1}(r)$ only from those $N$ for which $b_{k}(N) \approx (5/6) a_{k+1}$. Knowing the size of $b_k$ allows us to obtain good estimates for $\varepsilon_\ell,~\ell \leq k-1$, since the effect of $b_{k+1},\dots,b_{L-1}$ on these $\varepsilon_\ell$'s can be shown to be small -- all of that provided $a_{k+1}$ is ``large,'' so again we need to use the existence of large partial quotients. 

Assume that \eqref{approx6} can thus be decomposed into
\begin{equation} \label{approx7}
\sum_{(b_0,b_1,\dots,b_{k-1})} \prod_{\ell =0}^{k-1} \prod_{b=0}^{b_{\ell}-1} P_{q_{\ell}} \left( r, (-1)^{\ell} \frac{b/a_{\ell+1} + \varepsilon_{\ell}}{q_{\ell}} \right) \times \sum_{(b_k,\dots,b_{L-1})} \prod_{\ell =k}^{L-1} \prod_{b=0}^{b_{\ell}-1} P_{q_{\ell}} \left( r, (-1)^{\ell} \frac{b/a_{\ell+1} + \varepsilon_{\ell}}{q_{\ell}} \right) ,
\end{equation}
which we can write as $A_k(r) \times B_k(r)$. In a similar way, we can decompose $\mathbf{J}_{4_1}(r')$ into $A_k(r') \times B_k(r')$, and we need to study
$$
h(r) \approx \log \frac{A_k(r) B_k(r)}{A_k(r') B_k(r')}.
$$
Assume that $r \to \alpha$ and that accordingly $r' \to \alpha' = \{1/\alpha\}$. Note that $A_k$ is composed of Sudler products $P_{q_\ell}$ for $\ell \leq k-1$. In all these products we can replace $r$ by $\alpha$ up to a very small error, so that instead of the quotient $A_k(r) / A_k(r')$, which depends on $k$ as well as on $r$, we now have a quotient which depends only on $k$ and $\alpha$. It remains to show that this quotient converges as $k \to \infty$, which essentially means that the influence of the initial partial quotient $a_1$ contributing to $A_k(r)$ ``stabilizes'' as the length of the products increases. This will follow from an application of the Cauchy convergence criterion. Furthermore, it turns out that $B_k(r)/B_k(r') \to 1$ as $r \to \alpha$ and $r' \to \alpha'$ (somewhat surprisingly without the need for any further technical assumptions on the sequence $r$). This convergence crucially relies on the fact that the continued fraction expansion of $r'$ arises from that of $r$ by a simple shift to the left, which implies that the Ostrowski numeration systems generated by $r$ and $r'$ are structurally very similar. Thus we can relate a Sudler product $P_N(r)$ for some $N$ to a corresponding Sudler product $P_{N'}(r')$ for a suitable $N'$  in such a way that the Ostrowski expansion of $N'$ (with respect to $r'$) is obtained from the one of $N$ (with respect to $r$) by a simple shift to the left. In other words, there is an (asymptotic) shift-invariance of the Sudler product, when the same shift is applied to the (continued fraction expansion of the) argument as well as to the (Ostrowski expansion of the) index. This effect is captured in Proposition \ref{tailprop} below.

We will carry out the steps sketched above in the following order. In Section \ref{sec_56} we obtain a precise version of the observation that $P_N(r)^2$ contributes significantly to $\mathbf{J}_{4_1}(r)$ only if $b_k \approx (5/6) a_{k+1}$ whenever $a_{k+1}$ is large (Proposition \ref{local5/6prop}). In Section \ref{fact_s_section} we show that the Sudler product $P_N(r)$ can be factorized into two components associated with $(b_0, \dots, b_{k-1})$ and $(b_k,\dots,b_L)$ respectively, provided that $a_{k+1}$ is large and that $b_k \approx (5/6) a_{k+1}$ (Lemma \ref{sudlerfactorlemma}). In other words, this result takes care of the dependence caused by the presence of the $\varepsilon_\ell$'s. In Section \ref{fact_k_section} we factorize $\mathbf{J}_{4_1}$ as sketched in \eqref{approx7}, again assuming that $a_{k+1}$ is large, thus eliminating the dependence caused by the Ostrowski numeration system (Proposition \ref{kashaevfactorprop}). In Section \ref{tailsection} we consider the ``tail'' part of the factorization, i.e.\ the quotient $B_k(r)/B_k(r')$ in the terminology from above. In Section \ref{zag_sec} we combine these ingredients and settle the continuity problem. 

\section{Approximate factorization of $\mathbf{J}_{4_1}$}\label{factorsection}

\subsection{The local $5/6$-principle} \label{sec_56}

In our previous paper \cite{AB1} we observed that $P_N(\alpha )$ is maximized when the overwhelming majority of the Ostrowski digits $b_{\ell}(N)$ are close to the ``optimal'' value $(5/6)a_{\ell +1}$, and found the precise asymptotics of $P_N(\alpha )$ in terms of the deviation of the Ostrowski digits from this optimum. The main results of \cite{AB1}, however, apply only under the regularity condition $(1+\log a_k) \ll a_{k+1}$ on the partial quotients of $\alpha$; most crucially, this is not satisfied by almost all $\alpha$ in the sense of Lebesgue measure, and thus such a restriction would not allow us to arrive at Corollary \ref{cor_th1}. In this paper we prove a ``local'' form of the $5/6$-principle which is concerned with the effect of one particular digit taking a value away from the optimum; our result holds without any regularity condition on the partial quotients.
\begin{prop}[Local $5/6$-principle]\label{local5/6prop} Let $0 \le k <L$ be such that $a_{k+1} \ge 7$, and set $b_k^*:=\lfloor (5/6)a_{k+1} \rfloor$. Let $0 \le N <q_L$.
\begin{enumerate}
\item[(i)] If $b_{k+1}(N)<a_{k+2}$, then $N^*=N+(b_k^*-b_k(N))q_k$ satisfies
\[ \begin{split} \log P_{N^*} &(\alpha ) - \log P_N (\alpha ) \\ & \begin{split} \ge 0.2326 \frac{(b_k^*-b_k(N))^2}{a_{k+1}} -C \bigg( &\frac{|b_k^*-b_k(N)|}{a_{k+1}} \left( 1+\log \max_{1 \le m \le k} a_m \right) \\ &+ I_{\{ b_k(N) \le 1 \}} I_{\{ b_{k+1}(N)>0.99 a_{k+2} \}} \log a_{k+2} + \frac{1}{q_k^2} \bigg) \end{split} \end{split} \]
with a universal constant $C>0$.
\item[(ii)] If $b_{k+1}(N)=a_{k+2}$, then $N^*=N+b_k^* q_k - q_{k+1}$ satisfies
\[ \begin{split} \log &P_{N^*} (\alpha ) - \log P_N (\alpha ) \\ \ge &0.1615 a_{k+1} -C \left( 1 +\log \max_{1 \le m \le k} a_m +\log a_{k+2}+I_{\{ a_{k+2}=1 \}} I_{\{ b_{k+2}(N)>0.99a_{k+3} \}} a_{k+3} \right) \end{split} \]
with a universal constant $C>0$.
\end{enumerate}
We use the convention $\log \max_{1 \le m \le k} a_m =0$ if $k=0$.
\end{prop}

\begin{remark} Note that in (i), $N^*$ is obtained from $N$ by replacing the Ostrowski digit $b_k(N)$ by $b_k^*$. In (ii), the assumption $b_{k+1}(N)=a_{k+2}$ forces $b_k(N)=0$, and $N^*$ is obtained from $N$ by reducing the Ostrowski digit $b_{k+1}(N)=a_{k+2}$ to $a_{k+2}-1$, and increasing $b_k(N)=0$ to $b_k^*$. In both (i) and (ii), we arrive at a valid Ostrowski expansion for $N^*$.
\end{remark}

Simply put, Proposition \ref{local5/6prop} states a Gaussian upper bound to $P_N(\alpha)/P_{N^*}(\alpha)$ in terms of the deviation of the Ostrowski digit $b_k(N)$ from the optimum $b_k^*$. In (i) resp.\ (ii), we could have used any numerical value less than
\[ \frac{9 \V}{25 \pi} =0.23260748\dots \qquad \textrm{resp.} \qquad \frac{\V}{4 \pi} = 0.16153297\dots . \]
These values are actually sharp, although they will not play any special role here. The precise decay of $P_N(\alpha)/P_{N^*}(\alpha)$ is in fact non-Gaussian \cite{AB1}, but for our purposes an upper bound will suffice.

\begin{proof}[Proof of Proposition \ref{local5/6prop}] For the sake of readability, set $f(x)=|2 \sin (\pi x)|$. We give a detailed proof for $k \ge 1$, and indicate at the end how to modify the proof for $k=0$.

Let
\[ V_{\ell}(x):= \sum_{n=1}^{q_{\ell}-1} \sin \left( \pi \frac{n \| q_{\ell} \alpha \|}{q_{\ell}} \right) \cot \left( \pi \frac{n (-1)^{\ell} p_{\ell}+x}{q_{\ell}} \right) \]
be the cotangent sum first introduced in \cite{AB1}. Observe that $V_{\ell}(x)$ is decreasing on the interval $(-1,1)$. Lemma \ref{cotangentsumlemma} and summation by parts yield
\begin{equation}\label{vlxbound}
|V_{\ell}(x)| \ll q_{\ell} \| q_{\ell} \alpha \| \left( \frac{1}{1-|x|} + \log \max_{1 \le m \le \ell} a_m \right) , \qquad x \in (-1,1).
\end{equation}
A key result in our previous paper \cite[Proposition 12]{AB1} states that for any $1 \le \ell <L$ such that $b_{\ell} (N) \ge 1$,
\begin{equation}\label{logpqlestimate}
\begin{split} \sum_{b=0}^{b_{\ell}(N)-1} \log P_{q_{\ell}} \left( \alpha, (-1)^{\ell} \frac{b q_{\ell} \| q_{\ell} \alpha \| + \varepsilon_{\ell}(N)}{q_{\ell}} \right) = &\sum_{b=1}^{b_{\ell}(N)-1} \log f(b q_{\ell} \| q_{\ell} \alpha \| + \varepsilon_{\ell}(N)) \\ &+\sum_{b=0}^{b_{\ell}(N)-1} V_{\ell} (b q_{\ell} \| q_{\ell} \alpha \| + \varepsilon_{\ell}(N)) \\ &+\log (2 \pi (b_{\ell}(N) q_{\ell} \| q_{\ell} \alpha \| + \varepsilon_{\ell}(N))) \\ &+E_{\ell}(N), \end{split}
\end{equation}
where $E_{\ell}(N) \le C/(a_{\ell +1}q_{\ell})$. A porism of \cite[Proposition 12]{AB1} shows that if in addition $b_{\ell}(N)/a_{\ell +1}$ is bounded away from $1$ and $a_{\ell +1}$ is large enough (e.g.\ $b_{\ell}(N) \le (5/6)a_{\ell +1}$ and $a_{\ell +1} \ge 7$ suffice), then we also have the lower bound $E_{\ell}(N) \ge -C (1/a_{\ell +1} + 1/q_{\ell}^2)$.\\

\noindent\textbf{(i)} Assume that $b_{k+1}(N)<a_{k+2}$, and let $N^*=N+(b_k^*-b_k(N))q_k$. In particular, $b_k(N^*)=b_k^*$ and $b_{\ell}(N^*)=b_{\ell}(N)$ for all $\ell \neq k$. By the definition \eqref{epsilondef} of $\varepsilon_{\ell}$,
\begin{equation}\label{epsilon*-epsilon(i)}
\varepsilon_{\ell}(N^*)-\varepsilon_{\ell}(N)= (-1)^{k+\ell} q_{\ell} (b_k^*-b_k(N)) \| q_k \alpha \|, \qquad 0 \le \ell <k,
\end{equation}
and $\varepsilon_{\ell}(N^*)=\varepsilon_{\ell}(N)$ for $k \le \ell <L$. The product form \eqref{pnproductform} thus gives
\begin{equation}\label{sumlogpql(i)}
\begin{split} \log P_{N^*}(\alpha) - \log P_N(\alpha ) = &\left( \sum_{b=0}^{b_k^*-1} - \sum_{b=0}^{b_k(N)-1} \right) \log P_{q_k} \left( \alpha , (-1)^k \frac{b q_k \| q_k \alpha \| + \varepsilon_k(N)}{q_k} \right) \\ & +\sum_{\ell =0}^{k-1} \left( g_{\ell} (\varepsilon_{\ell}(N^*)) - g_{\ell}(\varepsilon_{\ell}(N)) \right) , \end{split}
\end{equation}
where
\[ \begin{split} g_{\ell}(x) := &\sum_{b=0}^{b_{\ell}(N)-1} \log P_{q_{\ell}} \left( \alpha, (-1)^{\ell} \frac{b q_{\ell} \| q_{\ell} \alpha \| + x}{q_{\ell}} \right) \\ = &\sum_{b=0}^{b_{\ell}(N)-1} \sum_{n=1}^{q_{\ell}} \log f \left( n \alpha + (-1)^{\ell} \frac{b q_{\ell} \| q_{\ell} \alpha \|+x}{q_{\ell}} \right) \\ = &\sum_{b=0}^{b_{\ell}(N)-1} \sum_{n=1}^{q_{\ell}} \log f \left( (n+bq_{\ell}) \alpha + (-1)^{\ell} \frac{x}{q_{\ell}} \right) \\ = &\sum_{n=1}^{b_{\ell}(N) q_{\ell}} \log f \left( n \alpha + (-1)^{\ell} \frac{x}{q_{\ell}} \right) . \end{split} \]

We claim that the second line in \eqref{sumlogpql(i)} is negligible. Since $b_{\ell}(N) q_{\ell}<q_{\ell +1}$, Lemma \ref{cotangentsumlemma} yields that for any real $|x|<q_{\ell} \| q_{\ell} \alpha \|$,
\[ \begin{split} |g_{\ell}'(x)| &= \left| \frac{(-1)^{\ell} \pi}{q_{\ell}} \sum_{n=1}^{b_{\ell}(N) q_{\ell}} \cot \left( \pi \left( n \alpha + (-1)^{\ell} \frac{x}{q_{\ell}} \right) \right) \right| \\ &\ll \frac{q_{\ell +1}}{q_{\ell}} \left( \frac{1}{1-\frac{|x|}{q_{\ell} \| q_{\ell} \alpha \|}} +\log \max_{1 \le m \le \ell +1} a_m \right) . \end{split} \]
Note that by Lemma \ref{epsilonlemma}, $|\varepsilon_{\ell} (N^*)|<q_{\ell} \| q_{\ell} \alpha \|$ and $|\varepsilon_{\ell} (N)|<q_{\ell} \| q_{\ell} \alpha \|$. Observing that $g_{\ell}'$ is decreasing, we have
\[ g_{\ell} (\varepsilon_{\ell}(N^*)) -g_{\ell} (\varepsilon_{\ell}(N)) = \int_{\varepsilon_{\ell}(N)}^{\varepsilon_{\ell}(N^*)} g_{\ell}' (x) \, \mathrm{d} x \ge - |g_{\ell}'(\varepsilon_{\ell}(N^*))| \cdot |\varepsilon_{\ell}(N^*) - \varepsilon_{\ell}(N) | \]
regardless of whether $\varepsilon_{\ell}(N^*)$ or $\varepsilon_{\ell}(N)$ is greater. From the previous two formulas and \eqref{epsilon*-epsilon(i)} we thus deduce
\begin{equation}\label{gleps*-gleps}
g_{\ell}(\varepsilon_{\ell}(N^*)) - g_{\ell} (\varepsilon_{\ell}(N)) \ge -C \frac{|b_k^*-b_k(N)|}{a_{k+1}} \cdot \frac{q_{\ell +1}}{q_k} \left( \frac{1}{1-\frac{|\varepsilon_{\ell}(N^*)|}{q_{\ell} \| q_{\ell} \alpha \|}} +\log \max_{1 \le m \le \ell +1} a_m \right) .
\end{equation}
First, let $0 \le \ell \le k-3$. Lemma \ref{epsilonlemma} implies that here
\[ 1-\frac{|\varepsilon_{\ell}(N^*)|}{q_{\ell} \| q_{\ell} \alpha \|} \ge \min \left\{ \frac{\| q_{\ell +1} \alpha \|}{\| q_{\ell} \alpha \|}, 1- \frac{\| q_{\ell +1} \alpha \|}{\| q_{\ell} \alpha \|} \right\} \ge \frac{\| q_{\ell +2} \alpha \|}{\| q_{\ell} \alpha \|} \gg \frac{q_{\ell +1}}{q_{\ell +3}} , \]
and we obtain
\[ \begin{split} \sum_{\ell =0}^{k-3} \left( g_{\ell}(\varepsilon_{\ell}(N^*)) - g_{\ell} (\varepsilon_{\ell}(N)) \right) &\ge -C \frac{|b_k^*-b_k(N)|}{a_{k+1}} \cdot \frac{1}{q_k} \sum_{\ell =0}^{k-3} q_{\ell +3} \left( 1+\log \max_{1 \le m \le \ell +1} a_m \right) \\ &\ge -C \frac{|b_k^*-b_k(N)|}{a_{k+1}} \left( 1+\log \max_{1 \le m \le k-2} a_m \right) . \end{split} \]
Next, consider the $\ell =k-2$ term. Since $b_k^* \le (5/6) a_{k+1}$, Lemma \ref{epsilonlemma} gives the better upper bound $\varepsilon_{k-2}(N^*) \le (17/18) q_{k-2} \| q_{k-1} \alpha \|$. Therefore
\[ 1-\frac{|\varepsilon_{k-2}(N^*)|}{q_{k-2} \| q_{k-2} \alpha \|} \ge \min \left\{ \frac{\| q_{k-1} \alpha \|}{\| q_{k-2} \alpha \|}, 1-\frac{17 \| q_{k-1} \alpha \|}{18 \| q_{k-2} \alpha \|} \right\} \ge \frac{\| q_{k-1} \alpha \|}{18 \| q_{k-2} \alpha \|} \gg \frac{q_{k-1}}{q_k} , \]
and \eqref{gleps*-gleps} leads to
\[ g_{k-2}(\varepsilon_{k-2}(N^*)) - g_{k-2} (\varepsilon_{k-2}(N)) \ge -C \frac{|b_k^*-b_k(N)|}{a_{k+1}} \left( 1+\log \max_{1 \le m \le k-1} a_m \right) . \]
Finally, consider the $\ell =k-1$ term. Lemma \ref{epsilonlemma} now gives the better lower bound $\varepsilon_{k-1}(N^*) \ge -(17/18)q_{k-1} \| q_{k-1} \alpha \|$. The assumption $a_{k+1} \ge 7$ ensures that $q_{k-1} \| q_k \alpha \| \le (17/18) q_{k-1} \| q_{k-1} \alpha \|$. Therefore
\[ 1-\frac{|\varepsilon_{k-1}(N^*)|}{q_{k-1} \| q_{k-1} \alpha \|} \ge \frac{1}{18}, \]
and \eqref{gleps*-gleps} yields
\[ g_{k-1}(\varepsilon_{k-1}(N^*)) - g_{k-1} (\varepsilon_{k-1}(N)) \ge -C \frac{|b_k^*-b_k(N)|}{a_{k+1}} \left( 1+\log \max_{1 \le m \le k} a_m \right) . \]
We have thus proved that
\[ \sum_{\ell =0}^{k-1} \left( g_{\ell} (\varepsilon_{\ell}(N^*)) - g_{\ell}(\varepsilon_{\ell}(N)) \right) \ge -C \frac{|b_k^*-b_k(N)|}{a_{k+1}} \left( 1+\log \max_{1 \le m \le k} a_m \right) , \]
and it remains to estimate the first line in \eqref{sumlogpql(i)}.

Applying \eqref{logpqlestimate} to both $N$ and $N^*$ gives
\[ \begin{split} \left( \sum_{b=0}^{b_k^*-1} - \sum_{b=0}^{b_k(N)-1} \right) \log P_{q_k} \bigg( \alpha, (-1)^{k} &\frac{b q_k \| q_k \alpha \| + \varepsilon_k (N)}{q_k} \bigg) \\ = &\left( \sum_{b=1}^{b_k^*-1} - \sum_{b=1}^{b_k(N)-1} \right) \log f(b q_k \| q_k \alpha \| + \varepsilon_k (N)) \\ &+\left( \sum_{b=0}^{b_k^*-1} - \sum_{b=0}^{b_k(N)-1} \right) V_k (b q_k \| q_k \alpha \| + \varepsilon_k (N)) \\ &+\log \frac{b_k^* q_k \| q_k \alpha \| + \varepsilon_k (N)}{b_k(N) q_k \| q_k \alpha \| + \varepsilon_k (N)} \\ &+E_k (N^*) - E_k(N) , \end{split} \]
provided that $b_k(N) \ge 1$; an obviously modified formula holds when $b_k(N)=0$. According to the porism mentioned after \eqref{logpqlestimate}, the assumption $a_{k+1} \ge 7$ ensures that $E_k(N^*)\ge -C (1/a_{k+1}+1/q_k^2)$, thus
\[ E_k (N^*) -E_k (N) \ge -C \left( \frac{1}{a_{k+1}} + \frac{1}{q_k^2} \right) . \]
It is also easy to see that
\[ \log \frac{b_k^* q_k \| q_k \alpha \| + \varepsilon_k (N)}{b_k(N) q_k \| q_k \alpha \| + \varepsilon_k (N)} \ge -C \frac{|b_k^*-b_k(N)|}{a_{k+1}} . \]
Formula \eqref{vlxbound} and the facts that $V_k(x)$ is decreasing and that $b_k^* q_k \| q_k \alpha \| + \varepsilon_k(N)$ is bounded away from $1$ show that
\[ \begin{split} \left( \sum_{b=0}^{b_k^*-1} - \sum_{b=0}^{b_k(N)-1} \right) V_k (b q_k \| q_k \alpha \| + \varepsilon_k (N)) &\ge -|b_k^*-b_k(N)| \cdot |V_k (b_k^* q_k \| q_k \alpha \| + \varepsilon_k (N))| \\ &\ge -C \frac{|b_k^*-b_k(N)|}{a_{k+1}} \left( 1+\log \max_{1 \le m \le k} a_m \right) \end{split} \]
regardless of whether $b_k(N)$ or $b_k^*$ is greater. Finally, we compare the sum of $\log f(b q_k \| q_k \alpha \| + \varepsilon_k (N))$ to the corresponding Riemann integral. Note that $b q_k \| q_k \alpha \| + \varepsilon_k (N)=b/a_{k+1}+O(1/a_{k+1})$, and in particular $b_k^* q_k \| q_k \alpha \| + \varepsilon_k (N) =5/6+O(1/a_{k+1})$.

Assume first, that $|b_k^*-b_k(N)| \le a_{k+1}/100$. Then for all $b$ between $b_k^*$ and $b_k(N)$, the points $b q_k \| q_k \alpha \| + \varepsilon_k (N)$ are bounded away from $0$ and $1$. Since $\log f(5/6)=0$, each term also satisfies $|\log f(b q_k \| q_k \alpha \| + \varepsilon_k )| \ll |b_k^*-b_k(N)|/a_{k+1}$. We thus obtain
\[ \begin{split} \left( \sum_{b=1}^{b_k^*-1} - \sum_{b=1}^{b_k(N)-1} \right) \log f(b q_k \| q_k \alpha \| &+ \varepsilon_k (N)) \\ &= a_{k+1} \int_{b_k(N)/a_{k+1}}^{b_k^*/a_{k+1}} \log f(x) \, \mathrm{d} x + O \left( \frac{|b_k^*-b_k(N)|}{a_{k+1}} \right) . \end{split} \]
The concavity of $\log f(x)$ implies that
\[ \inf_{\substack{y \in [0,1] \\ y \neq 5/6}} \frac{1}{(5/6-y)^2} \int_y^{5/6} \log f(x) \, \mathrm{d}x = \frac{1}{(5/6)^2} \int_0^{5/6} \log f(x) \, \mathrm{d}x = \frac{9 \V}{25 \pi} . \]
The first line in \eqref{sumlogpql(i)} thus satisfies
\[ \begin{split} \left( \sum_{b=0}^{b_k^*-1} \right. &- \left. \sum_{b=0}^{b_k(N)-1} \right) \log P_{q_k} \left( \alpha, (-1)^{k} \frac{b q_k \| q_k \alpha \| + \varepsilon_k (N)}{q_k} \right) \\ &\ge \frac{9\V}{25 \pi} \cdot \frac{(b_k^*-b_k(N))^2}{a_{k+1}} -C \left( \frac{|b_k^*-b_k(N)|}{a_{k+1}} \left( 1 +\log \max_{1 \le m \le k} a_m \right) + \frac{1}{q_k^2} \right) , \end{split} \]
and the claim follows.

Next, assume that $|b_k^*-b_k(N)|>a_{k+1}/100$. By Lemma \ref{epsilonlemma},
\[ \begin{split} b q_k \| q_k \alpha \| + \varepsilon_k (N) &\le (a_{k+1}-1) q_k \| q_k \alpha \| + q_k \| q_{k+1} \alpha \| \le 1-\frac{1}{a_{k+1}+2}, \\ b q_k \| q_k \alpha \| + \varepsilon_k (N) &\ge q_k \| q_{k+1} \alpha \| \gg \frac{1}{a_{k+1} a_{k+2}} . \end{split} \]
Hence each term satisfies $|\log f(b q_k \| q_k \alpha \| + \varepsilon_k (N))| \ll \log (a_{k+1} a_{k+2})$, and by comparing the sum to the corresponding integral we obtain
\[ \begin{split} \left( \sum_{b=1}^{b_k^*-1} - \sum_{b=1}^{b_k(N)-1} \right) \log f(b q_k \| q_k \alpha \| + \varepsilon_k (N)) &= a_{k+1} \int_{b_k(N)/a_{k+1}}^{b_k^*/a_{k+1}} \log f(x) \, \mathrm{d} x + O \left( \log (a_{k+1} a_{k+2}) \right) \\ &\ge \frac{9 \V}{25 \pi} \cdot \frac{(b_k^*-b_k(N))^2}{a_{k+1}} - C \log (a_{k+1} a_{k+2}) . \end{split} \]
If $b_k(N) \ge 2$, then the term $b=1$ does not appear in the previous formula, and we have the better lower bound
\[ b q_k \| q_k \alpha \| + \varepsilon_k (N) \ge 2 q_k \| q_k \alpha \| + \varepsilon_k (N) \gg \frac{1}{a_{k+1}} . \]
If $b_{k+1}(N) \le 0.99 a_{k+2}$, then Lemma \ref{epsilonlemma} gives the better lower bound
\[ b q_k \| q_k \alpha \| + \varepsilon_k (N) \ge 0.001 q_k \| q_k \alpha \| \gg \frac{1}{a_{k+1}} . \]
Therefore in these two cases the error term $\log a_{k+2}$ can be removed. The first line in \eqref{sumlogpql(i)} thus satisfies
\[ \begin{split} \left( \sum_{b=0}^{b_k^*-1} \right. -& \left. \sum_{b=0}^{b_k(N)-1} \right) \log P_{q_k} \left( \alpha, (-1)^{k} \frac{b q_k \| q_k \alpha \| + \varepsilon_k (N)}{q_k} \right) \\ \ge &\frac{9\V}{25 \pi} \cdot \frac{(b_k^*-b_k(N))^2}{a_{k+1}} \\ &-C \left( 1+ \log a_{k+1} + I_{\{ b_k(N) \le 1 \}} I_{\{ b_{k+1}(N) >0.99 a_{k+2} \}} \log a_{k+2} +\log \max_{1 \le m \le k} a_m + \frac{1}{q_k^2} \right) . \end{split} \]
After an arbitrarily small reduction in the value of $9\V /(25 \pi)$, the error term $\log a_{k+1}$ can be removed, and the claim follows. This finishes the proof of (i) when $k \ge 1$.\\

\noindent\textbf{(ii)} Assume that $b_{k+1}(N)=a_{k+2}$ (hence $b_k(N)=0$), and let $N^*=N+b_k^* q_k - q_{k+1}$. In particular, $b_k(N^*)=b_k^*$, $b_{k+1}(N^*)=a_{k+2}-1$ and $b_{\ell}(N^*)=b_{\ell}(N)$ for all $\ell \neq k,k+1$. By the definition \eqref{epsilondef} of $\varepsilon_{\ell}$,
\begin{equation}\label{epsilon*-epsilon(ii)}
\begin{split} \varepsilon_{\ell}(N^*) - \varepsilon_{\ell}(N) &= (-1)^{k+\ell} q_{\ell} (b_k^* \| q_k \alpha \| + \| q_{k+1} \alpha \| ) , \qquad 0 \le \ell \le k-1, \\ \varepsilon_k (N^*) - \varepsilon_k(N) &= q_k \| q_{k+1} \alpha \| , \end{split}
\end{equation}
and $\varepsilon_{\ell}(N^*)=\varepsilon_{\ell}(N)$ for $k+1 \le \ell <L$. The product form \eqref{pnproductform} thus gives
\begin{equation}\label{sumlogpql(ii)}
\begin{split} \log P_{N^*}(\alpha ) - &\log P_N (\alpha ) = \\ &-\log P_{q_{k+1}} \left( \alpha, (-1)^{k+1} \frac{(a_{k+2}-1) q_{k+1} \| q_{k+1} \alpha \| + \varepsilon_{k+1}(N)}{q_{k+1}} \right) \\ &+ \sum_{b=0}^{b_k^*-1} \log P_{q_k} \left( \alpha, (-1)^k \frac{b q_k \| q_k \alpha \| + \varepsilon_k (N^*)}{q_k} \right) \\ &+ \sum_{\ell =0}^{k-1} \left( g_{\ell} (\varepsilon_{\ell}(N^*)) - g_{\ell} (\varepsilon_{\ell}(N)) \right) , \end{split}
\end{equation}
where $g_{\ell}(x)$ is as before. The same arguments as in the proof of (i) yield
\[ \sum_{\ell =0}^{k-1} \left( g_{\ell} (\varepsilon_{\ell}(N^*)) - g_{\ell} (\varepsilon_{\ell}(N)) \right) \ge -C \left( 1+\log \max_{1 \le m \le k} a_m \right) \]
and
\[ \begin{split} \sum_{b=0}^{b_k^*-1} \log P_{q_k} &\left( \alpha, (-1)^k \frac{b q_k \| q_k \alpha \| + \varepsilon_k (N^*)}{q_k} \right) \\ &\ge \frac{\V}{4 \pi} a_{k+1} -C \left( 1+\log (a_{k+1} a_{k+2}) + \log \max_{1 \le m \le k} a_m + \frac{1}{q_k^2} \right) ,  \end{split} \]
and it remains to estimate the first line in \eqref{sumlogpql(ii)}.

A special case of a general estimate in our previous paper \cite[Proposition 11 (ii)]{AB1} states that the point $x=(a_{k+2}-1) q_{k+1} \| q_{k+1} \alpha \| +\varepsilon_{k+1}(N)$ satisfies
\[ \begin{split} \log P_{q_{k+1}} &\left( \alpha, (-1)^{k+1} \frac{x}{q_{k+1}} \right) \\ &\le \log \left( f \left( \| q_{k+1} \alpha \| +x/q_{k+1} \right) \frac{f(x)}{f(x/q_{k+1})} \right) + V_{k+1}(x) + \frac{C}{a_{k+2}^2 q_{k+1}} , \end{split} \]
with the convention that $f(x)/f(x/q_{k+1})=q_{k+1}$ in case $x=0$. Here
\[ \log \left( f \left( \| q_{k+1} \alpha \| +x/q_{k+1} \right) \frac{f(x)}{f(x/q_{k+1})} \right) \le \log \left( f \left( \| q_{k+1} \alpha \| +x/q_{k+1} \right) q_{k+1} \right) \le C. \]
Assume first, that $x \ge 0$. Then \eqref{vlxbound} and the fact that $V_{k+1}$ is decreasing give
\[ V_{k+1}(x) \le V_{k+1}(0) \le C \left( 1+\log \max_{1 \le m \le k+1} a_m \right) . \]
Next, assume that $x<0$. Then $a_{k+2}=1$. The general identity from the theory of continued fractions
\[ \frac{1}{q_{k+1} \| q_{k+1} \alpha \|} = [a_{k+2}; a_{k+3}, \dots ] + [0;a_{k+1}, a_k, \dots ] \ge 1+\frac{1}{a_{k+3}+1}  \]
and Lemma \ref{epsilonlemma} show that
\[ \varepsilon_{k+1} (N) \ge -q_{k+1} \| q_{k+1} \alpha \| \ge - \frac{a_{k+3}+1}{a_{k+3}+2} . \]
Formula \eqref{vlxbound} now gives
\[ V_{k+1}(x) \le C \left( \frac{1}{1+x} + \log \max_{1 \le m \le k+1} a_m \right) \le C \left( a_{k+3} + \log \max_{1 \le m \le k+1} a_m  \right) . \]
On the other hand, if $b_{k+2}(N) \le 0.99 a_{k+3}$, then Lemma \ref{epsilonlemma} yields the better lower bound $\varepsilon_{k+1}(N) \ge -0.999 q_{k+1} \| q_{k+1} \alpha \|$, and \eqref{vlxbound} similarly leads to
\[ V_{k+1}(x) \le C \left( 1 + \log \max_{1 \le m \le k+1} a_m \right) . \]
The first line in \eqref{sumlogpql(ii)} thus satisfies
\[ \begin{split} -\log P_{q_{k+1}} &\left( \alpha, (-1)^{k+1} \frac{(a_{k+2}-1) q_{k+1} \| q_{k+1} \alpha \| + \varepsilon_{k+1}(N)}{q_{k+1}} \right) \\ &\ge -C \left( 1 + I_{\{ a_{k+2}=1 \}} I_{\{ b_{k+2}(N)>0.99 a_{k+3} \}} a_{k+3} + \log \max_{1 \le m \le k+1} a_m  \right) . \end{split} \]
After an arbitrarily small reduction in the value of $\V /(4 \pi)$, the error term $\log a_{k+1}$ can be removed, and the claim follows. This finishes the proof of (ii) when $k \ge 1$.\\

We now indicate how to modify the proof for $k=0$. In (i) resp.\ (ii) formula \eqref{sumlogpql(i)} resp.\ \eqref{sumlogpql(ii)} still hold, with $\sum_{\ell=0}^{k-1} \left( g_{\ell} (\varepsilon_{\ell}(N^*)) - g_{\ell}(\varepsilon_{\ell}(N)) \right) =0$ being an empty sum. Since now $q_k=q_0=1$, we have $P_{q_k}(\alpha, x)=f(\alpha +x)$. Instead of applying \eqref{logpqlestimate} to $N$ and $N^*$, we can simply use
\[ \begin{split} \left( \sum_{b=0}^{b_k^*-1} - \sum_{b=0}^{b_k(N)-1} \right) &\log P_{q_k} \left( \alpha, (-1)^k \frac{bq_k \| q_k \alpha \| + \varepsilon_{k}(N)}{q_k} \right) \\ &= \left( \sum_{b=0}^{b_k^*-1} - \sum_{b=0}^{b_k(N)-1} \right) \log f ((b+1) \alpha + \varepsilon_0(N) ) , \end{split} \]
and compare the right hand side to the corresponding Riemann integral as in the case $k \ge 1$ above. This finishes the proof of (i) and (ii) when $k=0$.
\end{proof}

\begin{cor}\label{5/6corollary} Let $0 \le k <K \le L$ be such that
\begin{equation}\label{ak+1condition}
\frac{1+\log \max_{1 \le m \le k}a_m}{\sqrt{a_{k+1} \log (1+a_{k+1})}} \le A
\end{equation}
with a suitably small universal constant $A>0$, and set $b_k^*:=\lfloor (5/6)a_{k+1} \rfloor$. Then
\begin{equation}\label{cor1}
\sum_{\substack{0 \le N <q_K \\ |b_k(N)-b_k^*| \ge 10 \sqrt{a_{k+1} \log a_{k+1}}}} P_N (\alpha )^2 \le a_{k+1}^{-20} \sum_{0 \le N < q_K} P_N (\alpha )^2 ,
\end{equation}
and also
\begin{equation}\label{cor2}
\sum_{\substack{0 \le N <q_K \\ b_0(N)=b_1(N)=\cdots =b_{k-1}(N)=0 \\ |b_k(N)-b_k^*| \ge 10 \sqrt{a_{k+1} \log a_{k+1}}}} P_N (\alpha )^2 \le a_{k+1}^{-20} \sum_{\substack{0 \le N < q_K \\ b_0(N)=b_1(N)=\cdots =b_{k-1}(N)=0}} P_N (\alpha )^2 .
\end{equation}
\end{cor}

\begin{proof} Let $S:=\sum_{0 \le N <q_K} P_N(\alpha )^2$, and consider the sets
\[ \begin{split} H_{\ell}(b):= &\left\{ 0 \le N <q_K \, : \, b_{\ell}(N)=b \right\} , \\ H_{\ell ,\ell'}(b,c):= &\left\{ 0 \le N <q_K \, : \, b_{\ell}(N)=b, \,\, b_{\ell'}(N)=c \right\} . \end{split} \]
Let $B>0$ be a suitably small universal constant. We consider 3 cases depending on the sizes of $a_{k+2}$ and $a_{k+3}$.\\

\noindent\textbf{Case 1.} Assume that $a_{k+2}>B a_{k+1}$. For any $0.99 a_{k+2} < b \le a_{k+2}$, the map $N \mapsto N+(b_{k+1}^*-b)q_{k+1}$ is an injection from $H_{k+1}(b)$ to $H_{k+1}(b_{k+1}^*)$; here as before $b_{k+1}^*:=\lfloor (5/6)a_{k+2} \rfloor$. Choosing $A$ small enough in terms of $B$, condition \eqref{ak+1condition} ensures that $a_{k+2}$ dominates $\log \max_{1 \le m \le k+1} a_m$. Since $0.2326 \cdot 0.99^2>0.2279$, Proposition \ref{local5/6prop} (i) applied with $k+1$ thus shows that
\[ \sum_{N \in H_{k+1}(b)} P_N(\alpha )^2 \le \sum_{N \in H_{k+1}(b_{k+1}^*)} \exp (-0.2279 a_{k+2}) P_N (\alpha )^2 \le \exp (-0.2279 a_{k+2}) S . \]
Summing over $0.99 a_{k+2} < b \le a_{k+2}$ leads to
\[ \sum_{\substack{0 \le N < q_K \\ b_{k+1}(N)>0.99 a_{k+2}}} P_N (\alpha )^2 \le a_{k+2} \exp (-0.2279 a_{k+2}) S \le a_{k+1}^{-100} S . \]
For any $b,c$ with $|b-b_k^*| \ge 10 \sqrt{a_{k+1} \log a_{k+1}}$ and $0 \le c \le 0.99 a_{k+2}$, the map $N \mapsto N+(b_k^*-b)q_k$ is an injection from $H_{k,k+1}(b,c)$ to $H_{k,k+1}(b_k^*,c)$. Condition \eqref{ak+1condition} ensures that the main term
\[ 0.2326 \frac{(b_k^*-b)^2}{a_{k+1}} \ge 23.26 \log a_{k+1} \]
dominates the error term in Proposition \ref{local5/6prop} (i), thus
\[ \sum_{N \in H_{k,k+1}(b,c)} P_N (\alpha )^2 \le \sum_{N \in H_{k,k+1}(b_k^*,c)} a_{k+1}^{-23.26} P_N(\alpha )^2 . \]
Summing over $b,c$ leads to
\[ \sum_{\substack{0 \le N < q_K \\ |b_k(N)-b_k^*| \ge 10 \sqrt{a_{k+1} \log a_{k+1}} \\ b_{k+1}(N) \le 0.99 a_{k+2}}} P_N (\alpha )^2 \le a_{k+1}^{-22.26} S , \]
and \eqref{cor1} follows.\\

\noindent\textbf{Case 2.} Assume that $a_{k+2}\le B a_{k+1}$ and $a_{k+3}>B a_{k+1}$. Choosing $A$ small enough in terms of $B$ ensures that $a_{k+3}$ dominates $\log \max_{1 \le m \le k+2} a_m$. Proposition \ref{local5/6prop} (i) applied with $k+2$ now leads to
\[ \sum_{\substack{0 \le N <q_K \\ b_{k+2}(N)>0.99a_{k+3}}} P_N(\alpha )^2 \le a_{k+3} \exp (-0.2279 a_{k+3}) S \le a_{k+1}^{-100} S . \]
Since the error term $\log a_{k+2}$ is negligible, using Proposition \ref{local5/6prop} (i) and (ii) we similarly deduce
\[ \sum_{\substack{0 \le N <q_K \\ |b_k(N)-b_k^*| \ge 10 \sqrt{a_{k+1} \log a_{k+1}} \\ b_{k+2}(N) \le 0.99 a_{k+3}}} P_N(\alpha )^2 \le a_{k+1}^{-22.26} S , \]
and \eqref{cor1} follows.\\

\noindent\textbf{Case 3.} Assume that $a_{k+2} \le B a_{k+1}$ and $a_{k+3} \le B a_{k+1}$. Choosing $B$ small enough, the error terms $\log a_{k+2}$ and $a_{k+3}$ are now negligible. Proposition \ref{local5/6prop} (i) and (ii) directly give
\[ \sum_{\substack{0 \le N <q_K \\ |b_k(N)-b_k^*| \ge 10 \sqrt{a_{k+1} \log a_{k+1}}}} P_N(\alpha )^2 \le a_{k+1}^{-22.26} S, \]
as claimed in \eqref{cor1}.\\

A straightforward modification of the proof leads to \eqref{cor2}.
\end{proof}

\subsection{Factoring the Sudler product} \label{fact_s_section}

\begin{lem}\label{sudlerfactorlemma} Let $1 \le k <L$ be such that $a_{k+1} \ge 150$, and set $b_k^*:= \lfloor (5/6) a_{k+1} \rfloor$. Let $0 \le N<q_L$, and set $N_1:= \sum_{\ell =0}^{k-1} b_{\ell}(N) q_{\ell}$ and $N_2:= \sum_{\ell =k}^{L-1} b_{\ell}(N) q_{\ell}$. If $|b_k(N)-b_k^*| \le a_{k+1}/10$, then
\[ P_N(\alpha ) = P_{N_1}\left( \alpha, (-1)^k \frac{5/6}{q_k} \right) P_{N_2}(\alpha ) \exp \left( O \left( \frac{|b_k(N)-b_k^*|+1}{a_{k+1}} \left( 1+\log \max_{1 \le m \le k} a_m \right) \right) \right) \]
with a universal implied constant.
\end{lem}

\begin{remark} Note that $N_1$ contains the initial segment of the Ostrowski digits of $N$, and $N_2$ contains the final segment. Roughly speaking, the lemma says that we can decompose $P_N$ into two products $P_{N_1}$ and $P_{N_2}$ which only depend on the Ostrowski representation of $N$ up to digit $k-1$, resp.\ only on the Ostrowski representation from the $k$-th digit onwards; the error in this decomposition is small provided that $a_{k+1}$ is large and $b_k$ is close to $b_k^*$ . From Corollary \ref{5/6corollary} we know that whenever $a_{k+1}$ is large, then only those $N$ with $b_k \approx b_k^*$ make a significant contribution towards the value of $\mathbf{J}_{4_1}$. Together Corollary \ref{5/6corollary} and Lemma \ref{sudlerfactorlemma} will allow us to obtain the desired factorization of $\mathbf{J}_{4_1}$ in the next section.
\end{remark}

\begin{proof}[Proof of Lemma \ref{sudlerfactorlemma}] We write $f(x):= |2 \sin (\pi x)|$. Using
\[ N_2 \alpha \equiv (-1)^k \frac{b_k(N) q_k \| q_k \alpha \| + \varepsilon_k(N)}{q_k} \pmod{1}, \]
we deduce
\[ \begin{split} P_N (\alpha ) &= P_{N_2}(\alpha ) \prod_{n=1}^{N_1} f (n \alpha + N_2 \alpha ) \\ &= P_{N_2}(\alpha ) \prod_{n=1}^{N_1} f \left( n \alpha + (-1)^k \frac{b_k(N) q_k \| q_k \alpha \| + \varepsilon_k (N)}{q_k} \right) . \end{split} \]
It will thus be enough to estimate
\begin{equation}\label{prodR+Qn}
\frac{P_N(\alpha )}{P_{N_1}\left( \alpha, (-1)^k \frac{5/6}{q_k} \right) P_{N_2}(\alpha )} = \prod_{n=1}^{N_1} \frac{f \left( n \alpha + (-1)^k \frac{b_k(N) q_k \| q_k \alpha \| + \varepsilon_k (N)}{q_k} \right)}{f \left( n \alpha + (-1)^k \frac{5/6}{q_k} \right)} = \prod_{n=1}^{N_1} |1+R+Q_n|,
\end{equation}
where, by trigonometric identities,
\[ \begin{split} R:= &\cos \left( \pi (-1)^k \frac{b_k(N) q_k \| q_k \alpha \| -5/6 + \varepsilon_k(N)}{q_k} \right) -1, \\ Q_n := &\sin \left( \pi (-1)^k \frac{b_k(N) q_k \| q_k \alpha \| -5/6 + \varepsilon_k(N)}{q_k} \right) \cot \left( \pi \left( n \alpha + (-1)^k \frac{5/6}{q_k} \right) \right) .  \end{split} \]

Clearly,
\[ 0 \ge b_k^* q_k \| q_k \alpha \| - \frac{5}{6} \ge \frac{(5/6)a_{k+1}-1}{a_{k+1}+2} - \frac{5}{6} \ge - \frac{8/3}{a_{k+1}}, \]
hence the assumptions $a_{k+1} \ge 150$, $|b_k(N)-b_k^*| \le a_{k+1}/10$ and Lemma \ref{epsilonlemma} lead to
\[ \begin{split} \frac{|b_k(N) q_k \| q_k \alpha \| -5/6 + \varepsilon_k(N)|}{q_k} &\le \frac{|b_k(N)-b_k^*| q_k \| q_k \alpha \| + |b_k^* q_k \| q_k \alpha \| -5/6| + |\varepsilon_k(N)|}{q_k} \\ &\le \frac{|b_k(N)-b_k^*|+11/3}{a_{k+1} q_k} \\ &\le \frac{28}{225 q_k} . \end{split} \]
By the general inequality $1-\cos (\pi t) \le (\pi^2 /2) t^2$, we thus have
\[ |R| \le \frac{\pi^2}{2} \left( \frac{|b_k(N)-b_k^*|+11/3}{a_{k+1} q_k} \right)^2  \le \frac{\pi^2}{2} \left( \frac{28}{225 q_k} \right)^2 < \frac{0.08}{q_k^2} . \]
Observe also that for all $1 \le n \le N_1$,
\[ \left\| n \alpha + (-1)^k \frac{5/6}{q_k} \right\| = \left\| \frac{np_k}{q_k} + (-1)^k \frac{5/6+n \| q_k \alpha \|}{q_k} \right\| \ge  \left( \frac{1}{6} - \frac{1}{a_{k+1}} \right) \left\| \frac{np_k}{q_k} \right\| \ge \frac{4}{25} \left\| \frac{np_k}{q_k} \right\| . \]
Therefore by the general inequality $|\cot (\pi t)| \le 1/(\pi \| t \|)$,
\[ |Q_n| \le \pi \frac{|b_k(N)-b_k^*|+11/3}{a_{k+1} q_k} \cdot \frac{25}{4 \pi \| np_k/q_k \|} \le \frac{28}{225 q_k} \cdot \frac{25}{4 \| np_k/q_k \|} < 0.78 . \]

In particular, $|R+Q_n| \le 0.86$, and so each factor $1+R+Q_n$ in \eqref{prodR+Qn} is bounded away from zero. Using the fact that $e^{t-2t^2} \le 1+t \le e^t$ for $|t| \le 0.86$ we obtain
\[ \exp \left( \sum_{n=1}^{N_1} (R+Q_n) - \sum_{n=1}^{N_1} 2(R+Q_n)^2 \right) \le \prod_{n=1}^{N_1} |1+R+Q_n| \le \exp \left( \sum_{n=1}^{N_1} (R+Q_n) \right) . \]
Here
\[ N_1 |R| \ll \frac{|b_k(N)-b_k^*|+1}{a_{k+1}} \qquad \textrm{and} \qquad N_1 R^2 \ll \frac{|b_k(N)-b_k^*|+1}{a_{k+1}} \]
are negligible, and so is
\[ \sum_{n=1}^{N_1} Q_n^2 \ll \frac{|b_k(N)-b_k^*|+1}{a_{k+1}} \sum_{n=1}^{q_k-1} \frac{1}{q_k^2 \| np_k/q_k \|^2} \ll \frac{|b_k(N)-b_k^*|+1}{a_{k+1}} . \]
Finally, the fact $q_k \| q_{k-1} \alpha \| \ge 1-1/a_{k+1} \ge 149/150>5/6$ and Lemma \ref{cotangentsumlemma} yield

\[ \begin{split} \left| \sum_{n=1}^{N_1} Q_n \right| &\ll \frac{|b_k(N)-b_k^*|+1}{a_{k+1}q_k} \left| \sum_{n=1}^{N_1} \cot \left( \pi \left( n \alpha + (-1)^k \frac{5/6}{q_k} \right) \right) \right| \\ &\ll \frac{|b_k(N)-b_k^*|+1}{a_{k+1}} \left( 1 + \log \max_{1 \le m \le k} a_m \right) , \end{split} \]
and the claim follows.
\end{proof}

\subsection{Factoring $\mathbf{J}_{4_1}$} \label{fact_k_section}

We now prove the key result of this section. In the special case of $\alpha \in \mathbb{Q}$ and $K=L$, the following proposition states an approximate factorization of $\mathbf{J}_{4_1}(\alpha)$ with a negligible multiplicative error provided that $a_{k+1}$ dominates $a_1, \dots, a_k$. Observe that the first factor in this factorization depends only on $a_1, \dots, a_k$; this will play a crucial role in the proof of Theorem \ref{th1}.
\begin{prop}\label{kashaevfactorprop} Let $1 \le k<K \le L$ be such that
\[ \xi_k := \frac{\sqrt{\log (1+a_{k+1})}}{\sqrt{a_{k+1}}} \left( 1 +\log \max_{1 \le m \le k} a_m \right) \le A \]
with a suitably small universal constant $A>0$. Then
\[ \begin{split} \sum_{0 \le N <q_K} P_N(\alpha )^2 = &\left( \sum_{0 \le N <q_k} P_N \left( \frac{p_k}{q_k}, (-1)^k \frac{5/6}{q_k} \right)^2 \right) \left( \sum_{\substack{0 \le N < q_K \\ b_0(N)=b_1(N)=\cdots =b_{k-1}(N)=0}} P_N(\alpha )^2 \right) \\ &\times \left( 1+O(\xi_k) \right) \end{split} \]
with a universal implied constant.
\end{prop}

\begin{proof} Corollary \ref{5/6corollary} gives
\[ \sum_{0 \le N <q_K} P_N(\alpha )^2 = \left( 1 +O \left( a_{k+1}^{-20} \right) \right) \sum_{\substack{0 \le N < q_K \\ |b_k(N)-b_k^*|<10 \sqrt{a_{k+1} \log a_{k+1}}}} P_N(\alpha )^2 . \]
Let us now apply Lemma \ref{sudlerfactorlemma} to each term of the sum on the right hand side. Letting $N_1, N_2$ be as in Lemma \ref{sudlerfactorlemma}, observe that the map $N \mapsto (N_1,N_2)$ is a bijection from
\[ \left\{ 0 \le N < q_K \, : \, |b_k(N)-b_k^*| < 10 \sqrt{a_{k+1} \log a_{k+1}} \right\} \]
to the product set
\[ [0,q_k) \times \left\{ 0 \le N < q_K \, : \, b_0(N)=\cdots =b_{k-1}(N)=0, \,\,\, |b_k(N)-b_k^*|< 10 \sqrt{a_{k+1} \log a_{k+1}} \right\} . \]
This leads to the factorization
\[ \begin{split} \sum_{0 \le N <q_K} P_N(\alpha )^2 = &\left( \sum_{0 \le N <q_k} P_N \left( \alpha , (-1)^k \frac{5/6}{q_k} \right)^2 \right) \left( \sum_{\substack{0 \le N < q_K \\ b_0(N)=b_1(N)=\cdots =b_{k-1}(N)=0 \\ |b_k(N)-b_k^*|<10 \sqrt{a_{k+1} \log a_{k+1}}}} P_N(\alpha )^2 \right) \\ &\times \left( 1+O(\xi_k) \right) . \end{split} \]
Corollary \ref{5/6corollary} shows that the condition $|b_k(N)-b_k^*|<10 \sqrt{a_{k+1} \log a_{k+1}}$ can be removed from the second sum, and it remains to replace $\alpha$ by $p_k/q_k$ in the first sum on the right hand side.

The so-called transfer principle for shifted Sudler products \cite[Proposition 11 (i)]{AB1} shows that for any $0 \le N <q_k$,
\[ \log \frac{P_N(\alpha, (-1)^k (5/6)/q_k)}{P_N(p_k/q_k , (-1)^k (5/6)/q_k)} = \sum_{n=1}^N \sin \left( \pi \frac{n \| q_k \alpha \|}{q_k} \right) \cot \left( \pi \frac{n (-1)^k p_k+5/6}{q_k} \right) + O (a_{k+1}^{-2}) . \]
Lemma \ref{cotangentsumlemma} and summation by parts yield
\[ \left| \sum_{n=1}^N \sin \left( \pi \frac{n \| q_k \alpha \|}{q_k} \right) \cot \left( \pi \frac{n (-1)^k p_k+5/6}{q_k} \right) \right| \ll q_k \| q_k \alpha \| \left( 1+\log \max_{1 \le m \le k} a_m \right) \ll \xi_k . \]
Hence for all $0 \le N <q_k$,
\[ P_N \left( \alpha, (-1)^k \frac{5/6}{q_k} \right) = P_N \left( \frac{p_k}{q_k}, (-1)^k \frac{5/6}{q_k} \right) \left( 1+O (\xi_k ) \right) , \]
and the claim follows.
\end{proof}

\section{A tail estimate}\label{tailsection}

Throughout this section, let $0< \alpha <1$ be a real number, and let $\alpha':= \{ 1/\alpha \}$. We write their continued fraction expansions as in Section \ref{trig_section}: $\alpha=[0;a_1,a_2,\dots, a_L]$ and $\alpha'=[0;a_2,a_3,\dots, a_L]$ in the rational case, whereas $\alpha=[0;a_1,a_2, \dots ]$ and $\alpha'=[0;a_2,a_3,\dots ]$ in the irrational case. Let $p_{\ell}/q_{\ell}=[0;a_1,a_2,\dots , a_{\ell}]$ and $p_{\ell}'/q_{\ell}'=[0;a_2,a_3,\dots, a_{\ell}]$ be the convergents to $\alpha$ and $\alpha'$, respectively. In particular,
\begin{equation}\label{ql'pl'formula}
q_{\ell}'=p_{\ell} \qquad \textrm{and} \qquad p_{\ell}'=q_{\ell}-a_1 p_{\ell} .
\end{equation}
To any $0 \le N <q_L$ whose Ostrowski expansion $N=\sum_{\ell=0}^{L-1}b_{\ell}(N) q_{\ell}$ with respect to $\alpha$ satisfies $b_1(N)<a_2$, let us associate $0 \le N' < q_L'$ defined as
\[ N':=\sum_{\ell =1}^{L-1} b_{\ell}(N) q_{\ell}' . \]
Observe that this is a valid Ostrowski expansion of $N'$ with respect to $\alpha'$. We will apply the mapping $N \mapsto N'$ three times, namely once in the proof of Theorem \ref{quantitativecontinuitytheorem} and twice in the proof of Theorem \ref{hasymptoticstheorem}, and we will always ensure that it is only applied to numbers $N$ for which indeed $b_1(N) < a_2$ (the domain and codomain of the mapping will be chosen in such a way that the mapping is bijective).

The product form \eqref{pnproductform} then gives
\begin{equation}\label{pn/pn'productform}
\frac{P_N(\alpha )}{P_{N'}(\alpha' )} = \prod_{b=0}^{b_0(N)-1} P_{q_0} \left( \alpha, \frac{b q_0 \| q_0 \alpha \| + \varepsilon_0 (N)}{q_0} \right) \prod_{\ell =1}^{L-1} \prod_{b=0}^{b_{\ell}(N)-1} \frac{P_{q_{\ell}} \left( \alpha, (-1)^{\ell} \frac{b q_{\ell} \| q_{\ell} \alpha \| + \varepsilon_{\ell}(N)}{q_{\ell}} \right)}{P_{q_{\ell}'} \left( \alpha', (-1)^{\ell -1} \frac{b q_{\ell}' \| q_{\ell}' \alpha' \| + \varepsilon_{\ell}'(N)}{q_{\ell}'} \right)} ,
\end{equation}
where $\varepsilon_{\ell}'(N):=q_{\ell}' \sum_{m=\ell+1}^{L-1} (-1)^{\ell+m-1} b_m (N) \| q_m' \alpha' \|$. The main result of this section is a tail estimate for the previous formula.
\begin{prop}\label{tailprop} Let $1 \le \ell <L$, and assume the following two conditions:
\begin{enumerate}
\item[(i)] $a_{\ell +1} \le (q_{\ell}')^{1/100}$ or $b_{\ell}(N) \le 0.99 a_{\ell +1}$,
\item[(ii)] $a_{\ell +2} \le (q_{\ell +1}')^{1/100}$ or $b_{\ell +1}(N) \le 0.99 a_{\ell +2}$.
\end{enumerate}
Then
\[ \prod_{b=0}^{b_{\ell}(N)-1} \frac{P_{q_{\ell}} \left( \alpha, (-1)^{\ell} \frac{b q_{\ell} \| q_{\ell} \alpha \| + \varepsilon_{\ell}(N)}{q_{\ell}} \right)}{P_{q_{\ell}'} \left( \alpha', (-1)^{\ell -1} \frac{b q_{\ell}' \| q_{\ell}' \alpha' \| + \varepsilon_{\ell}'(N)}{q_{\ell}'} \right)} = \exp \left( O \left( \frac{(a_2+\cdots +a_{\ell})^{3/4}}{(q_{\ell}')^{3/4}} + \frac{\log (a_1+1)}{q_{\ell}'} \right) \right) \]
with a universal implied constant.
\end{prop}

We give the proof of Proposition \ref{tailprop} after two preliminary lemmas. As before we write $f(x):=|2 \sin (\pi x)|$.

\begin{lem}\label{explicitlemma} For any $x,x' \in \mathbb{R}$, we have the explicit formulas
\[ P_{q_{\ell}} \left( \alpha, (-1)^{\ell} \frac{x}{q_{\ell}} \right) = f \left( \| q_{\ell} \alpha \| + \frac{x}{q_{\ell}} \right) \frac{f(z)}{f(z/q_{\ell})} \prod_{n=1}^{q_{\ell}-1} \frac{f \left( \frac{n-y_n-z}{q_{\ell}} \right)}{f \left( \frac{n-z}{q_{\ell}} \right)} , \]
and similarly
\[ P_{q_{\ell}'} \left( \alpha', (-1)^{\ell -1} \frac{x'}{q_{\ell}'} \right) = f \left( \| q_{\ell}' \alpha' \| + \frac{x'}{q_{\ell}'} \right) \frac{f(z')}{f(z'/q_{\ell}')} \prod_{n=1}^{q_{\ell}'-1} \frac{f \left( \frac{n-y_n'-z'}{q_{\ell}'} \right)}{f \left( \frac{n-z'}{q_{\ell}'} \right)} , \]
where
\begin{equation}\label{ynyn'zz'}
\begin{split} y_n:= \left( \left\{ \frac{nq_{\ell -1}}{q_{\ell}} \right\} - \frac{1}{2} \right) q_{\ell} \| q_{\ell} \alpha \| \qquad &\textrm{and} \qquad y_n':= \left( \left\{ \frac{nq_{\ell -1}'}{q_{\ell}'} \right\} - \frac{1}{2} \right) q_{\ell}' \| q_{\ell}' \alpha' \| , \\ z:= x+\frac{q_{\ell} \| q_{\ell} \alpha \|}{2} \qquad &\textrm{and} \qquad z':=x'+\frac{q_{\ell}' \| q_{\ell}' \alpha' \|}{2} . \end{split}
\end{equation}
\end{lem}

\begin{proof} By peeling off the last factor,
\[ \begin{split} P_{q_{\ell}} &\left( \alpha, (-1)^{\ell} \frac{x}{q_{\ell}} \right) \\ &= f \left( q_{\ell} \alpha + (-1)^{\ell} \frac{x}{q_{\ell}} \right) \prod_{n=1}^{q_{\ell}-1} f \left( n \alpha + (-1)^{\ell} \frac{x}{q_{\ell}} \right) \\ &= f \left( \| q_{\ell} \alpha \| + \frac{x}{q_{\ell}} \right) \prod_{n=1}^{q_{\ell}-1} f \left( \frac{np_{\ell}}{q_{\ell}} + (-1)^{\ell} \left( \left\{ \frac{n}{q_{\ell}} \right\} - \frac{1}{2} \right) \| q_{\ell} \alpha \| + (-1)^{\ell} \frac{x+\frac{q_{\ell} \| q_{\ell} \|}{2}}{q_{\ell}} \right) . \end{split} \]
The mapping $n \mapsto n q_{\ell -1}$ is a bijection on the set of nonzero residues modulo $q_{\ell}$, which sends $n p_\ell$ to $(-1)^{\ell+1} n$ as a consequence of $p_{\ell-1} q_\ell - p_\ell q_{\ell-1} = (-1)^\ell$. Using this bijection to reorder the product in the previous formula, by the symmetry of $f$ we obtain
\[ P_{q_{\ell}} \left( \alpha, (-1)^{\ell} \frac{x}{q_{\ell}} \right) = f \left( \| q_{\ell} \alpha \| + \frac{x}{q_{\ell}} \right) \prod_{n=1}^{q_{\ell}-1} f \left( \frac{n-y_n-z}{q_{\ell}} \right) . \]
The simple identity \cite[Proposition 9]{AB1}
\[ \prod_{n=1}^{q-1} f \left( \frac{n-t}{q} \right) = \left\{ \begin{array}{ll} \frac{f(t)}{f(t/q)} & \textrm{if } t/q \not\in \mathbb{Z}, \\ q & \textrm{if } t/q \in \mathbb{Z} \end{array} \right. \qquad (q \in \mathbb{N}, \,\, t \in \mathbb{R}) \]
further shows that
\[ P_{q_{\ell}} \left( \alpha, (-1)^{\ell} \frac{x}{q_{\ell}} \right) = f \left( \| q_{\ell} \alpha \| + \frac{x}{q_{\ell}} \right) \frac{f(z)}{f(z/q_{\ell})} \prod_{n=1}^{q_{\ell}-1} \frac{f \left( \frac{n-y_n-z}{q_{\ell}} \right)}{f \left( \frac{n-z}{q_{\ell}} \right)} , \]
as claimed. The proof for $\alpha'$ is entirely analogous.
\end{proof}

\begin{lem}\label{ql-ql'lemma} Let $1 \le \ell \le m<L$. If either $m \ge 2$, or $m=1$ and $a_2>1$, then
\[ \big|q_{\ell} \| q_m \alpha \| - q_{\ell}' \| q_m' \alpha' \| \big| \le \frac{2}{q_{\ell +1} q_{m+1}'} . \]
\end{lem}

\begin{proof} The case $m=1$, $a_2>1$ can be checked ``by hand.'' Assume that $m \ge 2$, and let $r=[a_{m+1};a_{m+2}, \dots ]$. Using \eqref{ql'pl'formula} and the identities $\| q_m \alpha \| = 1/(rq_m+q_{m-1})$ and $\| q_m' \alpha' \|=1/(rq_m'+q_{m-1}')$ from the theory of continued fractions, we deduce
\[ q_{\ell} \| q_m \alpha \| - q_{\ell}' \| q_m' \alpha' \| = \frac{r q_{\ell} q_m \left( \frac{p_m}{q_m} - \frac{p_{\ell}}{q_{\ell}} \right) + q_{\ell} q_{m-1} \left( \frac{p_{m-1}}{q_{m-1}} - \frac{p_{\ell}}{q_{\ell}} \right)}{(rq_m+q_{m-1}) (rq_m'+q_{m-1}')} . \]
Note that $a_{m+1} \le r \le a_{m+1}+1$. If $m=\ell$, then
\[ \big| q_{\ell} \| q_m \alpha \| - q_{\ell}' \| q_m' \alpha' \| \big| = \frac{1}{(r q_{\ell} + q_{\ell -1})(rq_{\ell}'+q_{\ell-1}')} \le \frac{1}{q_{\ell +1} q_{\ell +1}'} , \]
as claimed. If $m=\ell +1$, then
\[ \big| q_{\ell} \| q_m \alpha \| - q_{\ell}' \| q_m' \alpha' \| \big| = \frac{r}{(r q_{\ell +1} + q_{\ell})(rq_{\ell +1}'+q_{\ell}')} \le \frac{1}{q_{\ell +1} q_{\ell +2}'} , \]
as claimed. If $m \ge \ell +2$, then somewhat roughly we have $|p_m/q_m - p_{\ell}/q_{\ell}| \le 2 |\alpha - p_{\ell}/q_{\ell}|$ and $|p_{m-1}/q_{m-1} - p_{\ell}/q_{\ell}| \le 2 |\alpha - p_{\ell}/q_{\ell}|$, which gives
\[ \big| q_{\ell} \| q_m \alpha \| - q_{\ell}' \| q_m' \alpha' \| \big| \le \frac{(r q_{\ell} q_m + q_{\ell} q_{m-1}) 2 |\alpha - p_{\ell}/q_{\ell}|}{(r q_m + q_{m-1}) (r q_m' + q_{m-1}')} = \frac{2 \| q_{\ell} \alpha \|}{r q_m' +q_{m-1}'} \le \frac{2}{q_{\ell +1} q_{m+1}'} , \]
as claimed.
\end{proof}

\begin{proof}[Proof of Proposition \ref{tailprop}] For any $0 \le b \le b_{\ell}(N)-1$, let
\[ x := b q_{\ell} \| q_{\ell} \alpha \| + \varepsilon_{\ell}(N) \qquad \textrm{and} \qquad x' := b q_{\ell}' \| q_{\ell}' \alpha' \| + \varepsilon_{\ell}'(N) ,  \]
and let $y_n,y_n',z,z'$ be as in \eqref{ynyn'zz'}; for the sake of readability, the dependence of $x,x',z,z'$ on $b$ is suppressed. The explicit formulas in Lemma \ref{explicitlemma} give
\begin{equation}\label{prodbexplicit}
\begin{split} \prod_{b=0}^{b_{\ell}(N)-1} \frac{P_{q_{\ell}} \left( \alpha, (-1)^{\ell} \frac{b q_{\ell} \| q_{\ell} \alpha \| + \varepsilon_{\ell}(N)}{q_{\ell}} \right)}{P_{q_{\ell}'} \left( \alpha', (-1)^{\ell -1} \frac{b q_{\ell}' \| q_{\ell}' \alpha' \| + \varepsilon_{\ell}'(N)}{q_{\ell}'} \right)} = & \prod_{b=0}^{b_{\ell}(N)-1} \frac{f \left( \| q_{\ell} \alpha \| + \frac{x}{q_{\ell}} \right) \frac{f(z)}{f(z/q_{\ell})}}{f \left( \| q_{\ell}' \alpha' \| + \frac{x'}{q_{\ell}'} \right) \frac{f(z')}{f(z'/q_{\ell}')}} \\ &\times \prod_{b=0}^{b_{\ell}(N)-1} \frac{\prod_{n=1}^{q_{\ell}-1}f \left( \frac{n-y_n-z}{q_{\ell}} \right) / f \left( \frac{n-z}{q_{\ell}} \right)}{\prod_{n=1}^{q_{\ell}'-1}f \left( \frac{n-y_n'-z'}{q_{\ell}'} \right) / f \left( \frac{n-z'}{q_{\ell}'} \right)} . \end{split}
\end{equation}
We may assume that either $\ell \ge 2$, or $\ell =1$ and $a_2>1$; indeed, otherwise the assumption $b_1(N)<a_2$ ensures that the left hand side of \eqref{prodbexplicit} is $1$, and we are done. Observe first, that by Lemma \ref{ql-ql'lemma} we have
\begin{equation}\label{x-x'}
\begin{split} |x-x'| &\le b \big|q_{\ell} \| q_{\ell} \alpha \| - q_{\ell}' \| q_{\ell}' \alpha' \| \big| + \left| \varepsilon_{\ell}(N) - \varepsilon_{\ell}'(N) \right| \\ &\le b \big|q_{\ell} \| q_{\ell} \alpha \| - q_{\ell}' \| q_{\ell}' \alpha' \| \big| + \sum_{m=\ell +1}^{L-1} b_m(N) \big|q_{\ell} \| q_m \alpha \| - q_{\ell}' \| q_m' \alpha' \| \big| \\ &\le \frac{2b}{q_{\ell +1} q_{\ell +1}'} + \sum_{m=\ell +1}^{L-1} a_{m+1} \frac{2}{q_{\ell +1} q_{m+1}'} \\ &\ll \frac{b+1}{q_{\ell +1} q_{\ell +1}'} . \end{split}
\end{equation}
Similarly, $|z-z'| \ll (b+1)/(q_{\ell +1} q_{\ell +1}')$. Lemma \ref{epsilonlemma} gives
\[ (b-1) q_{\ell} \| q_{\ell} \alpha \| + q_{\ell} \| q_{\ell +1} \alpha \| \le x \le q_{\ell} \| q_{\ell -1} \alpha \| - (a_{\ell +1} -b) q_{\ell} \| q_{\ell} \alpha \| , \]
and similar inequalities hold for $x'$.

We will use the following elementary estimate several times. Given a fixed parameter $0<\delta<1/2$, one readily checks that
\begin{equation}\label{expestimatet}
e^{t-4 ( \log 1/\delta ) t^2} \le 1+t \le e^t \qquad \textrm{for all } t \ge -1+\delta ,
\end{equation}
and consequently
\begin{equation}\label{expestimateU/V}
\exp \left( \frac{U-V}{V} - \left( 4 \log \frac{1}{\delta} \right) \left( \frac{U-V}{V} \right)^2 \right) \le \frac{U}{V} \le \exp \left( \frac{U-V}{V} \right) \qquad \textrm{for all } \frac{U}{V} \ge \delta .
\end{equation}

Let us now estimate the first line in \eqref{prodbexplicit}. Standard trigonometric identities and estimates yield
\[ \begin{split} f \left( \| q_{\ell} \alpha \| + \frac{x}{q_{\ell}} \right) &= f \left( \frac{z}{q_{\ell}} + \frac{\| q_{\ell} \alpha \|}{2} \right) \\ &= 2 \left| \sin \left( \pi \frac{z}{q_{\ell}} \right) + \cos \left( \pi \frac{z}{q_{\ell}} \right) \pi \frac{\| q_{\ell} \alpha \|}{2} \right| + O \left( \| q_{\ell} \alpha \|^2 f \left( \frac{z}{q_{\ell}} \right) + \| q_{\ell} \alpha \|^3 \right) , \end{split} \]
and so
\[ \begin{split} f \left( \| q_{\ell} \alpha \| + \frac{x}{q_{\ell}} \right) \frac{f(z)}{f(z/q_{\ell})} &= \left| 1 + \cot \left( \pi \frac{z}{q_{\ell}} \right) \pi \frac{\| q_{\ell} \alpha \|}{2} \right| f(z) + O \left( \| q_{\ell} \alpha \|^2 f(z) + \| q_{\ell} \alpha \|^3 q_{\ell} \right) \\ &= \big| x+q_{\ell} \| q_{\ell} \alpha \| \big| \frac{f(z)}{|z|} + O \left( \left( \| q_{\ell} \alpha \|^2 + \frac{\| q_{\ell} \alpha \|}{q_{\ell}^2} \right) f(z) + \| q_{\ell} \alpha \|^3 q_{\ell} \right) . \end{split} \]
We now turn the additive error into a multiplicative one. If $b \ge 1$, then by checking that $1 \le |x+q_{\ell} \| q_{\ell} \alpha \| | /|z| \le 2$, $f(z) \gg 1/a_{\ell +1}$ and
\[ f \left( \| q_{\ell} \alpha \| + \frac{x}{q_{\ell}} \right) \frac{f(z)}{f(z/q_{\ell})} \left( |x+q_{\ell} \| q_{\ell} \alpha \| | \frac{f(z)}{|z|} \right)^{-1} \gg \frac{1}{a_{\ell +1}} , \]
the estimate \eqref{expestimateU/V} with $\delta \approx 1/a_{\ell +1}$ gives
\[ f \left( \| q_{\ell} \alpha \| + \frac{x}{q_{\ell}} \right) \frac{f(z)}{f(z/q_{\ell})} = \big|x+q_{\ell} \| q_{\ell} \alpha \| \big| \frac{f(z)}{|z|} \exp \left( O \left( \frac{1}{a_{\ell +1}q_{\ell}^2} \right) \right) . \]
If $b=0$, then $f(z)/|z| \gg 1$, $f(z)\ll 1/a_{\ell +1}$ and
\[ f \left( \| q_{\ell} \alpha \| + \frac{x}{q_{\ell}} \right) \frac{f(z)}{f(z/q_{\ell})} \left( \big|x+q_{\ell} \| q_{\ell} \alpha \| \big| \frac{f(z)}{|z|} \right)^{-1} \gg 1 , \]
hence \eqref{expestimateU/V} with $\delta \approx 1$ gives
\[ f \left( \| q_{\ell} \alpha \| + \frac{x}{q_{\ell}} \right) \frac{f(z)}{f(z/q_{\ell})} = \big|x+q_{\ell} \| q_{\ell} \alpha \| \big| \frac{f(z)}{|z|} \exp \left( O \left( \frac{1}{a_{\ell +1}^2 q_{\ell}^2 (q_{\ell} \| q_{\ell} \alpha \| + \varepsilon_{\ell}(N))} \right) \right) . \]
Here $q_{\ell} \| q_{\ell} \alpha \| + \varepsilon_{\ell}(N) \gg 1/(a_{\ell +1} a_{\ell +2})$. On the other hand, if $b_{\ell +1}(N) \le 0.99 a_{\ell +2}$, then Lemma \ref{epsilonlemma} yields the better lower bound $q_{\ell} \| q_{\ell} \alpha \| + \varepsilon_{\ell}(N) \gg 1/a_{\ell +1}$. Combining all these cases and using assumption (ii), we thus have
\[ \begin{split} f \left( \| q_{\ell} \alpha \| + \frac{x}{q_{\ell}} \right) \frac{f(z)}{f(z/q_{\ell})} = & \big|x+q_{\ell} \| q_{\ell} \alpha \| \big| \frac{f(z)}{|z|} \\ &\times \exp \left( O \left( \frac{ I_{\{ b \ge 1 \}}}{a_{\ell +1}q_{\ell}^2} + I_{\{ b=0 \}} \frac{1+I_{\{ b_{\ell +1}(N) >0.99 a_{\ell +2} \} }a_{\ell +2}}{a_{\ell +1} q_{\ell}^2} \right) \right) \\ = &\big|x+q_{\ell} \| q_{\ell} \alpha \| \big| \frac{f(z)}{|z|} \exp \left( O \left( \frac{ I_{\{ b \ge 1 \}}}{a_{\ell +1}q_{\ell}^2} + \frac{ I_{\{ b=0 \}}}{q_{\ell +1}^{0.99} q_{\ell}} \right) \right) . \end{split} \]
An identical proof gives
\[ f \left( \| q_{\ell}' \alpha' \| + \frac{x'}{q_{\ell}'} \right) \frac{f(z')}{f(z'/q_{\ell}')} = \big|x'+q_{\ell}' \| q_{\ell}' \alpha' \| \big| \frac{f(z')}{|z'|} \exp \left( O \left( \frac{ I_{\{ b \ge 1 \}}}{a_{\ell +1}(q_{\ell}')^2} + \frac{ I_{\{ b=0 \}}}{(q_{\ell +1}')^{0.99} q_{\ell}'} \right) \right) , \]
therefore
\begin{equation}\label{firstline1}
\prod_{b=0}^{b_{\ell}(N)-1} \frac{f \left( \| q_{\ell} \alpha \| + \frac{x}{q_{\ell}} \right) \frac{f(z)}{f(z/q_{\ell})}}{f \left( \| q_{\ell}' \alpha' \| + \frac{x'}{q_{\ell}'} \right) \frac{f(z')}{f(z'/q_{\ell}')}} = \exp \left( O \left( \frac{1}{(q_{\ell}')^{1.99}} \right) \right) \prod_{b=0}^{b_{\ell}(N)-1} \frac{\big|x+q_{\ell} \| q_{\ell} \alpha \| \big| \frac{f(z)}{|z|}}{\big |x'+q_{\ell}' \| q_{\ell}' \alpha' \| \big| \frac{f(z')}{|z'|}} .
\end{equation}

Lemma \ref{ql-ql'lemma} and \eqref{x-x'} show that here
\[ x+q_{\ell} \| q_{\ell} \alpha \| = x'+q_{\ell}' \| q_{\ell}' \alpha' \| + O \left( \frac{b+1}{a_{\ell +1}^2q_{\ell} q_{\ell}'} \right) . \]
If $b \ge 1$, then $b/a_{\ell +1} \ll x+q_{\ell} \| q_{\ell} \alpha \| \ll b/a_{\ell +1}$ and the same holds for $\alpha'$, hence \eqref{expestimateU/V} with $\delta \approx 1$ gives
\[ \frac{\big|x+q_{\ell} \| q_{\ell} \alpha \|  \big|}{\big| x' +q_{\ell}' \| q_{\ell}' \alpha' \| \big|} = \exp \left( O \left( \frac{1}{a_{\ell +1} q_{\ell} q_{\ell}'} \right) \right) . \]
If $b=0$, then \eqref{expestimateU/V} with $\delta \approx I_{\{ b_{\ell +1}(N) \le 0.99 a_{\ell +2}\}} 1/a_{\ell +1} + 1/(a_{\ell +1} a_{\ell +2})$ and assumption (ii) yield
\[ \begin{split} &\frac{\big| x+q_{\ell} \| q_{\ell} \alpha \| \big|}{\big|x' +q_{\ell}' \| q_{\ell}' \alpha' \| \big|} \\ &= \exp \left( O \left( \frac{1}{a_{\ell +1}^2 q_{\ell} q_{\ell}' (q_{\ell} \| q_{\ell} \alpha \| + \varepsilon_{\ell}(N))} + \frac{1+\log a_{\ell +1} + I_{\{ b_{\ell +1}(N)>0.99 a_{\ell +2} \}} \log (a_{\ell+1}a_{\ell +2})}{a_{\ell +1}^4 q_{\ell}^2 (q_{\ell}')^2 (q_{\ell} \| q_{\ell} \alpha \| + \varepsilon_{\ell}(N))^2} \right) \right) \\ &= \exp \left( O \left( \frac{1+I_{\{ b_{\ell +1}(N) > 0.99 a_{\ell +2} \}}a_{\ell +2}}{a_{\ell +1}q_{\ell} q_{\ell}'} + \frac{\log a_{\ell +1} + I_{\{ b_{\ell +1}(N) > 0.99 a_{\ell +2} \}} a_{\ell +2}^2 \log (a_{\ell+1}a_{\ell +2})}{a_{\ell +1}^2 q_{\ell}^2 (q_{\ell}')^2} \right) \right) \\ &= \exp \left( O \left( \frac{1}{(q_{\ell}')^{1.99}} \right) \right) . \end{split} \]
The previous two formulas show that
\begin{equation}\label{firstline2}
\prod_{b=0}^{b_{\ell}(N)-1} \frac{\big|x+q_{\ell} \| q_{\ell} \alpha \| \big|}{\big|x' +q_{\ell}' \| q_{\ell}' \alpha' \| \big|} = \exp \left( O \left( \frac{1}{(q_{\ell}')^{1.99}} \right) \right) .
\end{equation}

Note that $z,z' \in (-1/2,1)$, and the function $f(t)/|t|$ is Lipschitz on $(-1/2,1)$. It is easy to see that $f(z)/|z| \gg (a_{\ell +1}-b)/a_{\ell +1}$, and \eqref{x-x'} also shows that
\[ \frac{f(z)}{|z|} = \frac{f(z')}{|z'|} + O \left( \frac{b+1}{a_{\ell +1}^2 q_{\ell} q_{\ell}'} \right) . \]
Estimate \eqref{expestimateU/V} with $\delta \approx 1/a_{\ell +1}$ thus gives
\[ \frac{f(z)/|z|}{f(z')/|z'|} = \exp \left( O \left( \frac{b+1}{(a_{\ell +1}-b) a_{\ell +1} q_{\ell} q_{\ell}'} + \frac{(b+1)^2 (1+\log a_{\ell +1})}{(a_{\ell +1}-b)^2 a_{\ell +1}^2 q_{\ell}^2 (q_{\ell}')^2} \right) \right) , \]
and using assumption (i) we get
\[ \prod_{b=0}^{b_{\ell}(N)-1} \frac{f(z)/|z|}{f(z')/|z'|} = \exp \left( O \left( \frac{1+I_{\{ b_{\ell}(N) > 0.99 a_{\ell +1} \}} \log a_{\ell +1}}{q_{\ell} q_{\ell}'}  \right) \right) = \exp \left( O \left( \frac{\log q_{\ell}'}{(q_{\ell}')^2} \right) \right) . \]
The previous formula, \eqref{firstline1} and \eqref{firstline2} show that the first line in \eqref{prodbexplicit} satisfies
\begin{equation}\label{firstline}
\prod_{b=0}^{b_{\ell}(N)-1} \frac{f \left( \| q_{\ell} \alpha \| + \frac{x}{q_{\ell}} \right) \frac{f(z)}{f(z/q_{\ell})}}{f \left( \| q_{\ell}' \alpha' \| + \frac{x'}{q_{\ell}'} \right) \frac{f(z')}{f(z'/q_{\ell}')}} = \exp \left( O \left( \frac{1}{(q_{\ell}')^{1.99}} \right) \right) .
\end{equation}

Next, we estimate the second line of \eqref{prodbexplicit}. We give a detailed proof of the case when $q_{\ell}' \ge 3$, and then indicate at the end how to modify the proof for $q_{\ell}' <3$.

Note that $|y_n| \le q_{\ell} \| q_{\ell} \alpha \|/2$, and $-1/2<z <1$. By standard trigonometric identities,
\begin{equation}\label{f/f}
\frac{f \left( \frac{n-y_n-z}{q_{\ell}} \right)}{f \left( \frac{n-z}{q_{\ell}} \right)} = \left| \cos \left( \pi \frac{y_n}{q_{\ell}} \right) - \sin \left( \pi \frac{y_n}{q_{\ell}} \right) \cot \left( \pi \frac{n-z}{q_{\ell}} \right) \right| .
\end{equation}
For any integer $|n| \ge 2$, here $\cos (\pi y_n/q_{\ell}) \ge \cos (\pi /6)$ and $|\sin (\pi y_n/q_{\ell}) \cot (\pi (n-z)/q_{\ell})| \le |y_n|/|n-z| \le 1/2$.

Let $1 \le \psi_{\ell} < q_{\ell}'/2$ be a parameter to be chosen. Applying \eqref{expestimatet} with $\delta \approx 1$ leads to
\[ \prod_{\psi_{\ell}<n \le q_{\ell}/2} \frac{f \left( \frac{n-y_n-z}{q_{\ell}} \right)}{f \left( \frac{n-z}{q_{\ell}} \right)} = \exp \left( - \sum_{\psi_{\ell}<n\le q_{\ell}/2} \sin \left( \pi \frac{y_n}{q_{\ell}} \right) \cot \left( \pi \frac{n-z}{q_{\ell}} \right) + O \left( \frac{1}{a_{\ell +1}^2 \psi_{\ell}} \right) \right) . \]
Since $q_{\ell -1}/q_{\ell}= [0;a_{\ell}, a_{\ell -1}, \dots , a_1]$, a classical estimate \cite[p.\ 126]{KN} states that the discrepancy of the sequence $\{ n q_{\ell -1}/q_{\ell} \}$, $1 \le n \le N$ is $\ll (N/q_{\ell}'+a_2+\cdots +a_{\ell})/N$. Koksma's inequality \cite[p.\ 143]{KN} thus yields
\[ \left| \sum_{n=1}^N \sin \left( \pi \left( \left\{ \frac{n q_{\ell -1}}{q_{\ell}} \right\} - \frac{1}{2} \right) \| q_{\ell} \alpha \| \right) \right| \ll (N/q_{\ell}'+a_2+\cdots +a_{\ell}) \| q_{\ell} \alpha \| \qquad (1 \le N \le q_{\ell}/2), \]
and summation by parts leads to
\[ \begin{split} \left| \sum_{\psi_{\ell}<n\le q_{\ell}/2} \sin \left( \pi \frac{y_n}{q_{\ell}} \right) \cot \left( \pi \frac{n-z}{q_{\ell}} \right) \right| &\ll q_{\ell} \| q_{\ell} \alpha \| \left( \frac{1+\log q_{\ell}}{q_{\ell}'} + \frac{a_2+\cdots +a_{\ell}}{\psi_{\ell}} \right) \\ &\ll \frac{\log (a_1+1)}{a_{\ell +1} q_{\ell}'} + \frac{a_2+\cdots +a_{\ell}}{a_{\ell +1} \psi_{\ell}} . \end{split} \]
Estimating the factors $-q_{\ell}/2<n<-\psi_{\ell}$ is entirely analogous, therefore we have
\begin{equation}\label{secondline1}
\prod_{n=1}^{q_{\ell}-1} \frac{f \left( \frac{n-y_n-z}{q_{\ell}} \right)}{f \left( \frac{n-z}{q_{\ell}} \right)} = \exp \left( O \left( \frac{\log (a_1+1)}{a_{\ell +1} q_{\ell}'} + \frac{a_2+\cdots +a_{\ell}}{a_{\ell +1} \psi_{\ell}} \right) \right)  \prod_{0<|n| \le \psi_{\ell}} \frac{f \left( \frac{n-y_n-z}{q_{\ell}} \right)}{f \left( \frac{n-z}{q_{\ell}} \right)}.
\end{equation}
An identical proof gives
\begin{equation}\label{secondline2}
\prod_{n=1}^{q_{\ell}'-1} \frac{f \left( \frac{n-y_n'-z'}{q_{\ell}'} \right)}{f \left( \frac{n-z'}{q_{\ell}'} \right)} = \exp \left( O \left( \frac{a_2+\cdots +a_{\ell}}{a_{\ell +1} \psi_{\ell}} \right) \right)  \prod_{0<|n| \le \psi_{\ell}} \frac{f \left( \frac{n-y_n'-z'}{q_{\ell}'} \right)}{f \left( \frac{n-z'}{q_{\ell}'} \right)}.
\end{equation}
We emphasize that we use the same cutoff $\psi_{\ell}$ for both $\alpha$ and $\alpha'$.

Consider now the ratio of the factors $2 \le |n| \le \psi_{\ell}$ in the previous two formulas. Formula \eqref{f/f} and its analogue for $\alpha'$ are still bounded away from zero. Therefore \eqref{expestimateU/V} with $\delta \approx 1$ gives
\[ \begin{split} &\frac{f \left( \frac{n-y_n-z}{q_{\ell}} \right) / f \left( \frac{n-z}{q_{\ell}} \right)}{f \left( \frac{n-y_n'-z'}{q_{\ell}'} \right) / f \left( \frac{n-z'}{q_{\ell}'} \right)} \\ &= \frac{\cos \left( \pi \frac{y_n}{q_{\ell}} \right) - \sin \left( \pi \frac{y_n}{q_{\ell}} \right) \cot \left( \pi \frac{n-z}{q_{\ell}} \right)}{\cos \left( \pi \frac{y_n'}{q_{\ell}'} \right) - \sin \left( \pi \frac{y_n'}{q_{\ell}'} \right) \cot \left( \pi \frac{n-z'}{q_{\ell}'} \right) } \\ &= \exp \left( O \left( \left| \sin \left( \pi \frac{y_n}{q_{\ell}} \right) \cot \left( \pi \frac{n-z}{q_{\ell}} \right) - \sin \left( \pi \frac{y_n'}{q_{\ell}'} \right) \cot \left( \pi \frac{n-z'}{q_{\ell}'} \right) \right| + \frac{1}{a_{\ell +1}^2 (q_{\ell}')^2} \right) \right) . \end{split} \]
Standard trigonometric estimates and \eqref{x-x'} show that here
\[ \begin{split} \bigg| \sin \left( \pi \frac{y_n}{q_{\ell}} \right) \cot \left( \pi \frac{n-z}{q_{\ell}} \right) &- \sin \left( \pi \frac{y_n'}{q_{\ell}'} \right) \cot \left( \pi \frac{n-z'}{q_{\ell}'} \right) \bigg| \\ &= \left| \frac{y_n}{n-z} - \frac{y_n'}{n-z'} \right| + O \left( \frac{1}{a_{\ell +1}^3 (q_{\ell}')^2} + \frac{n^2}{a_{\ell +1} (q_{\ell}')^3} \right) \\ &\ll \frac{|y_n-y_n'|}{|n|} + \frac{b+1}{a_{\ell +1}^3 n^2 q_{\ell} q_{\ell}'}+ \frac{1}{a_{\ell +1}^3 (q_{\ell}')^2} + \frac{n^2}{a_{\ell +1} (q_{\ell}')^3} . \end{split} \]
Observe that the function $(\{ n t \} -1/2)/|n|$ consists of linear segments of slope $\pm 1$ with jumps at the points $j/n$, $j \in \mathbb{Z}$. Recalling \eqref{ql'pl'formula}, we have
\[ \left| \frac{q_{\ell -1}}{q_{\ell}} - \frac{q_{\ell -1}'}{q_{\ell}'} \right| = \frac{1}{q_{\ell} q_{\ell}'} , \]
whereas the distance from $q_{\ell -1}/q_{\ell}$ to any jump $j/n$ is
\[ \left| \frac{q_{\ell -1}}{q_{\ell}} - \frac{j}{n} \right| \ge \frac{1}{|n| q_{\ell}} . \]
Hence there is no jump between $q_{\ell -1}/q_{\ell}$ and $q_{\ell-1}'/q_{\ell}'$, and using also Lemma \ref{ql-ql'lemma} we obtain
\[ \frac{|y_n-y_n'|}{|n|} \le \left| \frac{q_{\ell -1}}{q_{\ell}} - \frac{q_{\ell -1}'}{q_{\ell}'} \right| q_{\ell} \| q_{\ell} \alpha \| + \frac{1}{2|n|} |q_{\ell} \| q_{\ell} \alpha \| - q_{\ell}' \| q_{\ell}' \alpha' \| | \ll \frac{1}{a_{\ell +1} q_{\ell} q_{\ell}'} . \]
After discarding negligible error terms, the previous estimates yield
\begin{equation}\label{secondline3}
\begin{split} \prod_{2 \le |n| \le \psi_{\ell}} \frac{f \left( \frac{n-y_n-z}{q_{\ell}} \right) / f \left( \frac{n-z}{q_{\ell}} \right)}{f \left( \frac{n-y_n'-z'}{q_{\ell}'} \right) / f \left( \frac{n-z'}{q_{\ell}'} \right)} &= \prod_{2 \le |n| \le \psi_{\ell}} \exp \left( O \left( \frac{1}{a_{\ell +1} (q_{\ell}')^2} + \frac{n^2}{a_{\ell +1} (q_{\ell}')^3} \right) \right) \\ &= \exp \left( O \left( \frac{\psi_{\ell}}{a_{\ell +1} (q_{\ell}')^2} + \frac{\psi_{\ell}^3}{a_{\ell +1} (q_{\ell}')^3} \right) \right) . \end{split}
\end{equation}

Finally, let $n= \pm 1$. Under the assumptions
\begin{equation}\label{cdconditions}
\begin{split} \frac{q_{\ell} \| q_{\ell} \alpha \| /2}{|n-z|} \le 1-\frac{1}{c},& \qquad \frac{q_{\ell}' \| q_{\ell}' \alpha' \| /2}{|n-z'|} \le 1-\frac{1}{c}, \\ |n-z| \ge \frac{1}{d},& \qquad |n-z'| \ge \frac{1}{d} \end{split}
\end{equation}
with some $c \ge 2$ and $d \ge 1$, it is not difficult to see that $1/c \ll f \left( \frac{n-y_n-z}{q_{\ell}} \right) / f \left( \frac{n-z}{q_{\ell}} \right) \ll 1$, and the same holds for $\alpha'$. Following the steps in the previous paragraph, we also deduce
\[ \begin{split} &\left| \frac{f \left( \frac{n-y_n-z}{q_{\ell}} \right)}{f \left( \frac{n-z}{q_{\ell}} \right)} - \frac{f \left( \frac{n-y_n'-z'}{q_{\ell}'} \right)}{f \left( \frac{n-z'}{q_{\ell}'} \right)} \right| \\ &= \Bigg| \left| \cos \left( \pi \frac{y_n}{q_{\ell}} \right) - \sin \left( \pi \frac{y_n}{q_{\ell}} \right) \cot \left( \pi \frac{n-z}{q_{\ell}} \right) \right| - \left| \cos \left( \pi \frac{y_n'}{q_{\ell}'} \right) - \sin \left( \pi \frac{y_n'}{q_{\ell}'} \right) \cot \left( \pi \frac{n-z'}{q_{\ell}'} \right) \right| \Bigg| \\ &\ll \frac{d}{a_{\ell +1} (q_{\ell}')^2} + \frac{d^2}{a_{\ell +1}^2 (q_{\ell}')^2} . \end{split} \]
Estimate \eqref{expestimateU/V} with $\delta \approx 1/c$ thus gives
\begin{equation}\label{npm1}
\frac{f \left( \frac{n-y_n-z}{q_{\ell}} \right) / f \left( \frac{n-z}{q_{\ell}} \right)}{f \left( \frac{n-y_n'-z'}{q_{\ell}'} \right) / f \left( \frac{n-z'}{q_{\ell}'} \right)} = \exp \left( O \left( \frac{cd}{a_{\ell +1} (q_{\ell}')^2} + \frac{cd^2}{a_{\ell +1}^2 (q_{\ell}')^2} + \frac{c^2 d^2 \log c}{a_{\ell +1}^2 (q_{\ell}')^4} + \frac{c^2 d^4 \log c}{a_{\ell +1}^4 (q_{\ell}')^4} \right) \right) .
\end{equation}

Consider the factor $n=-1$. If $b \ge 1$, then $z>0$; in particular, \eqref{cdconditions} is satisfied with $c=2$ and $d=1$. If $b=0$, then the general fact
\[ q_{\ell} \| q_{\ell} \alpha \| \le \frac{1}{a_{\ell +1} + \frac{1}{a_{\ell +2}+1}} \le 1-\frac{1}{a_{\ell +2}+2} \]
and its analogue for $\alpha'$ show that \eqref{cdconditions} is satisfied with $c=a_{\ell+2}+2$ and $d=2$. On the other hand, if $b_{\ell +1} (N) \le 0.99 a_{\ell +2}$, then by Lemma \ref{epsilonlemma} we have the better lower bound $\varepsilon_{\ell}(N) \ge -0.999 q_{\ell} \| q_{\ell} \alpha \|$, and consequently
\[ z = \frac{q_{\ell} \| q_{\ell} \alpha \|}{2} + \varepsilon_{\ell}(N) \ge -0.499 q_{\ell} \| q_{\ell} \alpha \| ; \]
in particular, \eqref{cdconditions} holds with $c \ll 1$ and $d=2$. Combining all these cases, \eqref{npm1} gives
\[ \begin{split} \frac{f \left( \frac{-1-y_{-1}-z}{q_{\ell}} \right) / f \left( \frac{-1-z}{q_{\ell}} \right)}{f \left( \frac{-1-y_{-1}'-z'}{q_{\ell}'} \right) / f \left( \frac{-1-z'}{q_{\ell}'} \right)} = \exp \bigg( O \bigg( &\frac{I_{\{ b \ge 1 \}}}{a_{\ell +1} (q_{\ell}')^2} + I_{\{ b=0 \}} \frac{1+I_{\{ b_{\ell +1} (N) >0.99 a_{\ell +2} \}} a_{\ell +2}}{a_{\ell +1} (q_{\ell}')^2} \\ &+ I_{\{ b=0 \}} \frac{I_{\{ b_{\ell +1}(N)>0.99 a_{\ell +2} \}} a_{\ell +2}^2 \log a_{\ell +2}}{a_{\ell +1}^2 (q_{\ell}')^4} \bigg) \bigg) . \end{split} \]
Simplifying the error using assumption (ii) shows
\begin{equation}\label{secondline4}
\prod_{b=0}^{b_{\ell}(N)-1} \frac{f \left( \frac{-1-y_{-1}-z}{q_{\ell}} \right) / f \left( \frac{-1-z}{q_{\ell}} \right)}{f \left( \frac{-1-y_{-1}'-z'}{q_{\ell}'} \right) / f \left( \frac{-1-z'}{q_{\ell}'} \right)} = \exp \left( O \left( \frac{1}{(q_{\ell}')^{1.99}} \right) \right) .
\end{equation}

Consider the factor $n=1$. By Lemma \ref{epsilonlemma}, we have
\[ |1-z| = 1-\left( \left( b+\frac{1}{2} \right) q_{\ell} \| q_{\ell} \alpha \| + \varepsilon_{\ell}(N) \right) \ge q_{\ell -1} \| q_{\ell} \alpha \| + \left( a_{\ell +1} -b-\frac{1}{2} \right) q_{\ell} \| q_{\ell} \alpha \| \gg \frac{1}{a_{\ell +1}} , \]
and
\[ \frac{q_{\ell} \| q_{\ell} \alpha \| /2}{|1-z|} \le \frac{1/2}{\frac{q_{\ell -1}}{q_{\ell}} + a_{\ell +1}-b-\frac{1}{2}} \le \frac{1}{\frac{2 q_{\ell -1}}{q_{\ell}}+1} . \]
Repeating the same estimates for $\alpha'$, we see that \eqref{cdconditions} is satisfied with $c \ll a_{\ell}$ and $d \ll a_{\ell +1}$. On the other hand, if $b_{\ell}(N) \le 0.99 a_{\ell +1}$, then \eqref{cdconditions} is satisfied with $c \ll 1$ and $d \ll 1$. Combining these cases, \eqref{npm1} gives
\[ \frac{f \left( \frac{1-y_1-z}{q_{\ell}} \right) / f \left( \frac{1-z}{q_{\ell}} \right)}{f \left( \frac{1-y_1'-z'}{q_{\ell}'} \right) / f \left( \frac{1-z'}{q_{\ell}'} \right)} = \exp \left( O \left( \frac{1}{a_{\ell +1} (q_{\ell}')^2} + \frac{I_{\{ b_{\ell} (N) > 0.99 a_{\ell +1} \} }}{q_{\ell -1}' q_{\ell}'} \right) \right) . \]
Simplifying the error using assumption (i) shows
\begin{equation}\label{secondline5}
\prod_{b=0}^{b_{\ell}(N)-1} \frac{f \left( \frac{1-y_1-z}{q_{\ell}} \right) / f \left( \frac{1-z}{q_{\ell}} \right)}{f \left( \frac{1-y_1'-z'}{q_{\ell}'} \right) / f \left( \frac{1-z'}{q_{\ell}'} \right)} = \exp \left( O \left( \frac{1}{q_{\ell -1}' (q_{\ell}')^{0.99}} \right) \right) .
\end{equation}

By \eqref{secondline1}--\eqref{secondline3}, \eqref{secondline4} and \eqref{secondline5} the second line of \eqref{prodbexplicit} satisfies
\[ \begin{split} \prod_{b=0}^{b_{\ell}(N)-1} &\frac{\prod_{n=1}^{q_{\ell}-1}f \left( \frac{n-y_n-z}{q_{\ell}} \right) / f \left( \frac{n-z}{q_{\ell}} \right)}{\prod_{n=1}^{q_{\ell}'-1}f \left( \frac{n-y_n'-z'}{q_{\ell}'} \right) / f \left( \frac{n-z'}{q_{\ell}'} \right)} \\ &= \exp \left( O \left( \frac{\log (a_1+1)}{q_{\ell}'} + \frac{a_2+\cdots +a_{\ell}}{\psi_{\ell}} + \frac{\psi_{\ell}}{(q_{\ell}')^2} + \frac{\psi_{\ell}^3}{(q_{\ell}')^3} + \frac{1}{q_{\ell -1}' (q_{\ell}')^{0.99}} \right) \right) . \end{split} \]
The optimal choice is
\[ \psi_{\ell} \approx (a_2+\cdots +a_{\ell})^{1/4} (q_{\ell}')^{3/4} , \]
which is easily seen to satisfy the required bounds $1 \le \psi_{\ell} < q_{\ell}'/2$. Discarding negligible error terms, we finally obtain
\[ \prod_{b=0}^{b_{\ell}(N)-1} \frac{\prod_{n=1}^{q_{\ell}-1}f \left( \frac{n-y_n-z}{q_{\ell}} \right) / f \left( \frac{n-z}{q_{\ell}} \right)}{\prod_{n=1}^{q_{\ell}'-1}f \left( \frac{n-y_n'-z'}{q_{\ell}'} \right) / f \left( \frac{n-z'}{q_{\ell}'} \right)} = \exp \left( O \left( \frac{\log (a_1+1)}{q_{\ell}'} +  \frac{(a_2+\cdots +a_{\ell})^{3/4}}{(q_{\ell}')^{3/4}} \right) \right) . \]
This finishes the estimation of the second line of \eqref{prodbexplicit}, and the proof of the proposition in the case when $q_{\ell}' \ge 3$.

We now indicate how to modify the estimate of the second line of \eqref{prodbexplicit} if $q_{\ell}' <3$. First of all note that if $q_{\ell}=2$, then
\[ y_1= \left( \left\{ \frac{q_{\ell -1}}{q_{\ell}} \right\} - \frac{1}{2} \right) q_{\ell} \| q_{\ell} \alpha \| =0 , \]
since we necessarily have $q_{\ell -1}=1$. Similarly, $q_{\ell}'=2$ implies $y_1'=0$. In particular, $\prod_{n=1}^{q_{\ell}'-1} f \left( \frac{n-y_n'-z'}{q_{\ell}'} \right) / f \left( \frac{n-z'}{q_{\ell}'} \right) =1$, as the product is either empty, or consists of the single factor $n=1$. We may thus assume that $q_{\ell} \ge 3$, otherwise the second line of \eqref{prodbexplicit} equals $1$, and we are done.

Following the steps above with $\psi_{\ell}=1$, as an analogue of \eqref{secondline1} we deduce
\[ \prod_{n=1}^{q_{\ell}-1} \frac{f \left( \frac{n-y_n-z}{q_{\ell}} \right)}{f \left( \frac{n-z}{q_{\ell}} \right)} = \exp \left( O \left( \frac{\log (a_1+1)}{a_{\ell +1}} \right) \right) \prod_{n \in \{ \pm 1 \}} \frac{f \left( \frac{n-y_n-z}{q_{\ell}} \right)}{f \left( \frac{n-z}{q_{\ell}} \right)} . \]
Clearly,
\[ f \left( \frac{n-y_n-z}{q_{\ell}} \right) = f \left( \frac{n-z}{q_{\ell}} \right) + O \left( \frac{1}{a_{\ell +1}q_{\ell}} \right) . \]
Under the assumptions
\[ \frac{q_{\ell} \| q_{\ell} \alpha \| /2}{|n-z|} \le 1-\frac{1}{c} \qquad \textrm{and} \qquad |n-z| \ge \frac{1}{d} \]
with some $c \ge 2$ and $d \ge 1$, we have $f \left( \frac{n-z}{q_{\ell}} \right) \gg 1/(dq_{\ell})$ and $f \left( \frac{n-y_n-z}{q_{\ell}} \right) / f \left( \frac{n-z}{q_{\ell}} \right) \gg 1/c$, thus \eqref{expestimateU/V} gives
\[ \frac{f \left( \frac{n-y_n-z}{q_{\ell}} \right)}{f \left( \frac{n-z}{q_{\ell}} \right)} = \exp \left( O \left( \frac{d}{a_{\ell +1}} + \frac{d^2 \log c}{a_{\ell +1}^2} \right) \right) . \]
Choosing $c,d$ optimally for $n=-1$ as above, we deduce
\[ \frac{f \left( \frac{-1-y_{-1}-z}{q_{\ell}} \right)}{f \left( \frac{-1-z}{q_{\ell}} \right)} = \exp \left( O \left( I_{\{ b \ge 1 \}}\frac{1}{a_{\ell +1}} + I_{\{ b=0 \}} \left( \frac{1}{a_{\ell +1}} + \frac{I_{\{ b_{\ell +1}(N) >0.99 a_{\ell +2} \}}a_{\ell +2}}{a_{\ell +1}^2} \right) \right) \right) , \]
hence by assumption (ii),
\[ \prod_{b=0}^{b_{\ell}(N)-1} \frac{f \left( \frac{-1-y_{-1}-z}{q_{\ell}} \right)}{f \left( \frac{-1-z}{q_{\ell}} \right)} = \exp \left( O(1) \right) . \]
Choosing $c,d$ optimally for $n=1$ as above, we deduce
\[ \frac{f \left( \frac{1-y_1-z}{q_{\ell}} \right)}{f \left( \frac{1-z}{q_{\ell}} \right)} = \exp \left( O \left( \frac{1}{a_{\ell +1}} + I_{\{ b_{\ell}(N)>0.99 a_{\ell +1} \}} (1+\log a_{\ell}) \right) \right) , \]
hence by assumption (i),
\[ \prod_{b=0}^{b_{\ell}(N)-1} \frac{f \left( \frac{1-y_1-z}{q_{\ell}} \right)}{f \left( \frac{1-z}{q_{\ell}} \right)} = \exp \left( O \left( 1+\log a_{\ell} \right) \right) . \]
Here either $\ell =1$, or $a_{\ell} \le q_{\ell}' <3$. Combining the previous estimates, the second line of \eqref{prodbexplicit} thus satisfies
\[ \prod_{b=0}^{b_{\ell}(N)-1} \frac{\prod_{n=1}^{q_{\ell}-1}f \left( \frac{n-y_n-z}{q_{\ell}} \right) / f \left( \frac{n-z}{q_{\ell}} \right)}{\prod_{n=1}^{q_{\ell}'-1}f \left( \frac{n-y_n'-z'}{q_{\ell}'} \right) / f \left( \frac{n-z'}{q_{\ell}'} \right)} = \exp \left( O \left( 1+\log a_1 \right) \right) . \]
This finishes the estimation of the second line of \eqref{prodbexplicit} and the proof of the proposition when $q_{\ell}'<3$.
\end{proof}

\section{Zagier's function $h(x)$}  \label{zag_sec}

\subsection{Continuity} \label{cont_section}

Let $\alpha=[0;a_1,a_2,\dots ]$ be irrational with convergents $p_{\ell}/q_{\ell}=[0;a_1, a_2, \dots, a_{\ell}]$, and let
\[ \xi_k := \frac{\sqrt{\log (1+a_{k+1})}}{\sqrt{a_{k+1}}} \left( 1+\log \max_{1 \le m \le k} a_m \right). \]
Let
\[ I_{k+1} := \left\{ [0;c_1,c_2, \dots ] \, : \, c_m = a_m \textrm{ for all } 1 \le m \le k+1 \right\} \]
denote the set of real numbers in $[0,1]$ whose first $k+1$ partial quotients are the same as those of $\alpha$. Recall that $I_{k+1}$ is an interval with rational endpoints $p_{k+1}/q_{k+1}$ and $(p_{k+1}+p_k)/(q_{k+1}+q_k)$, so for every $k$ the irrational $\alpha$ is an interior point of $I_{k+1}$.
\begin{thm}\label{quantitativecontinuitytheorem} Let $k \ge A^{-1} \log \log (a_1+2)$ be such that $\xi_k \le A$ with a suitably small universal constant $A>0$. Then
\[ \sup_{r \in I_{k+1} \cap \mathbb{Q}} h(r) - \inf_{r \in I_{k+1} \cap \mathbb{Q}} h(r) \ll \xi_k + \frac{(a_2+\cdots +a_k)^{3/4}}{(q_k/a_1)^{3/4}} + \frac{\log (a_1+1)}{q_k/a_1} \]
with a universal implied constant.
\end{thm}
Note that here the last two terms in the upper bound converge to $0$:
\[ \lim_{k \to \infty} \left( \frac{(a_2+\cdots +a_k)^{3/4}}{(q_k/a_1)^{3/4}} + \frac{\log (a_1+1)}{q_k/a_1} \right) =0 . \]
If $\sup_{k \ge 1} a_k =\infty$, then $\xi_k \to 0$ along a suitable subsequence, consequently $h$ can be extended to $\mathbb{Q} \cup \{ \alpha \}$ so that it is continuous at $\alpha$. This establishes Zagier's continuity conjecture at all non-badly approximable irrationals, and proves Theorem \ref{th1}. If $\alpha$ is badly approximable, then letting $a:= \limsup_{k \to \infty} a_k$ denote the largest integer which appears in its continued fraction expansion infinitely many times, Theorem \ref{quantitativecontinuitytheorem} gives
\[ \limsup_{x \to \alpha} h(x) - \liminf_{x \to \alpha} h(x) \ll \frac{(\log a)^{3/2}}{\sqrt{a}} , \]
a weaker form of the conjecture which falls short of implying continuity at $\alpha$.

\begin{proof}[Proof of Theorem \ref{quantitativecontinuitytheorem}] Let $\alpha':= \{ 1/\alpha \}=[0;a_2,a_3,\dots]$. Let $r=[0;c_1,c_2, \dots, c_L] \in I_{k+1} \cap \mathbb{Q}$ be arbitrary with convergents $\overline{p}_{\ell}/\overline{q}_{\ell}=[0;c_1,c_2, \dots, c_{\ell}]$, and define $r':= \{ 1/r \}=[0;c_2,c_3,\dots ,c_L]$ with convergents $\overline{p}_{\ell}'/\overline{q}_{\ell}'=[0;c_2,c_3, \dots, c_{\ell}]$. By the assumption $r \in I_{k+1}$, we have $\overline{p}_{\ell}/\overline{q}_{\ell} =p_{\ell}/q_{\ell}$ and $\overline{p}_{\ell}'/\overline{q}_{\ell}' = p_{\ell}'/q_{\ell}'$ for all $\ell \le k+1$.

Let us apply Proposition \ref{kashaevfactorprop} to $r$ and to $r'$:
\[ \begin{split} \sum_{0 \le N <\overline{q}_L} P_N(r)^2 = &\left( \sum_{0 \le N <q_k} P_N \left( \frac{p_k}{q_k}, (-1)^k \frac{5/6}{q_k} \right)^2 \right) \left( \sum_{\substack{0 \le N < \overline{q}_L \\ b_0(N)=b_1(N)=\cdots =b_{k-1}(N)=0}} P_N(r)^2 \right) \\ &\times (1+O(\xi_k )) , \end{split} \]
and
\[ \begin{split} \sum_{0 \le N <\overline{q}_L'} P_N(r')^2 = &\left( \sum_{0 \le N <q_k'} P_N \left( \frac{p_k'}{q_k'}, (-1)^{k-1} \frac{5/6}{q_k'} \right)^2 \right) \left( \sum_{\substack{0 \le N < \overline{q}_L' \\ b_1'(N)=b_2'(N)=\cdots =b_{k-1}'(N)=0}} P_N(r')^2 \right) \\ &\times (1+O(\xi_k )) . \end{split} \]
Note that here $N=\sum_{\ell =0}^{L-1} b_{\ell}(N) \overline{q}_{\ell}$ denotes the Ostrowski expansion of an integer $0 \le N <\overline{q}_L$ with respect to $r$, whereas $N=\sum_{\ell =1}^{L-1} b_{\ell}'(N) \overline{q}_{\ell}'$ denotes the Ostrowski expansion of an integer $0 \le N <\overline{q}_L'$ with respect to $r'$. Hence
\begin{equation}\label{hrestimate}
h(r) = \log \frac{\sum_{0 \le N <\overline{q}_L} P_N(r)^2}{\sum_{0 \le N <\overline{q}_L'} P_N(r')^2} = M_k(\alpha ) + \log \frac{\displaystyle{\sum_{\substack{0 \le N < \overline{q}_L \\ b_0(N)=b_1(N)=\cdots =b_{k-1}(N)=0}} P_N(r)^2}}{\displaystyle{\sum_{\substack{0 \le N < \overline{q}_L' \\ b_1'(N)=b_2'(N)=\cdots =b_{k-1}'(N)=0}} P_N(r')^2}} + O (\xi_k ) ,
\end{equation}
where
\[ M_k(\alpha ) := \log \frac{\sum_{0 \le N <q_k} P_N \left( \frac{p_k}{q_k}, (-1)^k \frac{5/6}{q_k} \right)^2}{\sum_{0 \le N <q_k'} P_N \left( \frac{p_k'}{q_k'}, (-1)^{k-1} \frac{5/6}{q_k'} \right)^2} . \]
The crucial observation is that $M_k(\alpha )$ does not depend on $r$, but only on the first $k$ partial quotients which are the same throughout $I_{k+1}$. It remains to estimate the second term in \eqref{hrestimate} uniformly in $r$.

Let $H$ denote the set of all integers $0 \le N < \overline{q}_L$ such that $b_0(N)=b_1(N)=\cdots =b_{k-1}(N)=0$, and $b_{\ell} (N) \le 0.99 c_{\ell +1}$ for every $k \le \ell \le L-1$ such that $c_{\ell +1} > (\overline{q}_{\ell}')^{1/100}$. Similarly, let $H'$ denote the set of all integers $0 \le N < \overline{q}_L'$ such that $b_1'(N)=b_2'(N)=\cdots =b_{k-1}'(N)=0$, and $b_{\ell}' (N) \le 0.99 c_{\ell +1}$ for every $k \le \ell \le L-1$ such that $c_{\ell +1} > (\overline{q}_{\ell}')^{1/100}$. Note that $0.2326 \cdot (5/6-0.99)^2>0.005$, and that by choosing $A$ small enough, $1+\log \max_{1 \le m \le \ell} c_m$ is negligible compared to $(\overline{q}_{\ell}')^{1/100}$. Applying Proposition \ref{local5/6prop} to all $k \le \ell \le L-1$ such that $c_{\ell +1}>(\overline{q}_{\ell}')^{1/100}$ thus shows that
\[ \begin{split} \sum_{\substack{0 \le N < \overline{q}_L \\ b_0(N)=b_1(N)=\cdots =b_{k-1}(N)=0}} P_N(r)^2 &= \left( 1+O \left( \sum_{\substack{k \le \ell \le L-1 \\ c_{\ell +1}> (\overline{q}_{\ell}')^{1/100}}} c_{\ell +1} e^{-0.005 c_{\ell +1}} \right) \right) \sum_{N \in H} P_N (r)^2 \\ &= \left( 1+O \left( q_k' e^{-0.005 (q_k')^{1/100}} \right) \right) \sum_{N \in H} P_N (r)^2 , \end{split} \]
and similarly
\[ \sum_{\substack{0 \le N < \overline{q}_L' \\ b_1'(N)=b_2'(N)=\cdots =b_{k-1}'(N)=0}} P_N(r')^2 = \left( 1+O \left( q_k' e^{-0.005 (q_k')^{1/100}} \right) \right) \sum_{N \in H'} P_N (r')^2 . \]
The map $N=\sum_{\ell=0}^{L-1}b_{\ell}(N) \overline{q}_{\ell} \mapsto N':= \sum_{\ell =1}^{L-1} b_{\ell}(N) \overline{q}_{\ell}'$, as introduced in Section \ref{tailsection}, is a bijection from $H$ to $H'$ (note that for $N \in H$ we have $b_1(N)= 0$, so the mapping $N \mapsto N'$ is indeed well-defined). By the product form \eqref{pn/pn'productform} and Proposition \ref{tailprop}, for any $N \in H$ we have
\[ \begin{split} \frac{P_N(r)}{P_{N'}(r')} &= \prod_{\ell =k}^{L-1} \prod_{b=0}^{b_{\ell}(N)-1} \frac{P_{\overline{q}_{\ell}} \left( r, (-1)^{\ell} \frac{b \overline{q}_{\ell} \| \overline{q}_{\ell} r \| + \overline{\varepsilon}_{\ell}(N)}{\overline{q}_{\ell}} \right)}{P_{\overline{q}_{\ell}'} \left( r', (-1)^{\ell -1} \frac{b \overline{q}_{\ell}' \| \overline{q}_{\ell}' r' \| + \overline{\varepsilon}_{\ell}'(N)}{\overline{q}_{\ell}'} \right)} \\ &= \prod_{\ell =k}^{L-1} \exp \left( O \left( \frac{(c_2+\cdots +c_{\ell})^{3/4}}{(\overline{q}_{\ell}')^{3/4}} + \frac{\log (c_1+1)}{q_{\ell}'} \right) \right) \\ &= \exp \left( O \left( \frac{(a_2+\cdots +a_k)^{3/4}}{(q_k')^{3/4}} + \frac{\log (a_1+1)}{q_k'} \right) \right) . \end{split} \]
Consequently,
\[ \begin{split} \frac{\displaystyle{\sum_{\substack{0 \le N < \overline{q}_L \\ b_0(N)=b_1(N)=\cdots =b_{k-1}(N)=0}} P_N(r)^2}}{\displaystyle{\sum_{\substack{0 \le N < \overline{q}_L' \\ b_1'(N)=b_2'(N)=\cdots =b_{k-1}'(N)=0}} P_N(r')^2}} &= \left( 1+O \left( q_k' e^{-0.005 (q_k')^{1/100}} \right) \right) \frac{\sum_{N \in H} P_N(r)^2}{\sum_{N \in H'} P_N(r')^2} \\ &= \exp \left( O \left( \frac{(a_2+\cdots +a_k)^{3/4}}{(q_k')^{3/4}} + \frac{\log (a_1+1)}{q_k'} \right) \right) , \end{split} \]
and \eqref{hrestimate} simplifies to
\[ h(r)=M_k(\alpha ) + O \left( \xi_k + \frac{(a_2+\cdots +a_k)^{3/4}}{(q_k')^{3/4}} + \frac{\log (a_1+1)}{q_k'} \right) . \]
Here $q_k' \gg q_k /a_1$, and $M_k(\alpha )$ does not depend on $r$, leading to the upper bound for the oscillation of $h(r)$ on $I_{k+1} \cap \mathbb{Q}$ in the claim.
\end{proof}

\subsection{Asymptotics} \label{asy_section}

\begin{proof}[Proof of Theorem \ref{hasymptoticstheorem}] Let $r$ be a rational in $(0,1)$, and let $r =[0;a_1,a_2,\dots, a_L]$ be its continued fraction expansion, with convergents $p_{\ell}/q_{\ell}=[0;a_1,a_2,\dots , a_{\ell}]$. Let $r':=\{ 1/r \}=[0;a_2,a_3,\dots, a_L]$, with convergents $p_{\ell}'/q_{\ell}'=[0;a_2,a_3,\dots, a_{\ell}]$. The Ostrowski expansion of $0 \le N <q_L$ with respect to $r$ will be written as $N=\sum_{\ell =0}^{L-1} b_{\ell}(N) q_{\ell}$, while that of $0 \le N <q_L'$ with respect to $r'$ as $N=\sum_{\ell =1}^{L-1} b_{\ell}'(N) q_{\ell}'$. The claim of the theorem can be equivalently stated as
\[ h(r) = \frac{\V}{2 \pi} a_1 + O (\log (a_1+1)) . \]

Let $D>0$ be a suitably large universal constant. Let $R$ denote the set of all integers $0 \le N <q_L$ such that $I_{\{ a_2>a_1 +D \}} b_1(N) \le 0.99 a_2$ and $I_{\{ a_2=1 \}}I_{\{ a_3>a_1 +D \}} b_2(N) \le 0.99 a_3$. Applying Corollary \ref{5/6corollary} with $k=1$ and $k=2$ yields the rough estimate
\[ \sum_{0 \le N < q_L} P_N(r)^2 \ll \sum_{N \in R} P_N(r)^2 . \]
Letting $R'$ denote the set of all integers $0 \le N <q_L'$ such that $I_{\{ a_2>a_1+D \}} b_1'(N) \le 0.99 a_2$ and $I_{\{ a_2=1 \}}I_{\{ a_3 > a_1+D \}} b_2'(N) \le 0.99 a_3$, we similarly deduce
\[ \sum_{0 \le N < q_L'} P_N(r')^2 \ll \sum_{N \in R'} P_N(r')^2 , \]
therefore
\begin{equation}\label{hrestimate1}
h(r) = \log \frac{\sum_{0 \le N <q_L}P_N(r)^2}{\sum_{0 \le N <q_L'}P_N(r')^2} = \log \frac{\sum_{N \in R} P_N(r)^2}{\sum_{N \in R'}P_N(r')^2} +O(1) .
\end{equation}
Set $b_0^*:=\lfloor (5/6) a_1 \rfloor$ and $f(x)=|2 \sin (\pi x)|$, as usual, and let
\[ H_{\ell}(c):= \left\{ 0 \le N < q_L \, : \, b_{\ell}(N)=c \right\} . \]
We distinguish between the cases $a_1>D$ and $a_1 \le D$.\\

\noindent\textbf{Case 1.} Assume that $a_1>D$. Corollary \ref{5/6corollary} shows that
\[ \sum_{N \in R} P_N(r)^2 \ll \sum_{\substack{N \in R \\ |b_0(N) -b_0^*|<10 \sqrt{a_1 \log a_1}}} P_N(r)^2 . \]
Hence in the numerator of \eqref{hrestimate1} it is enough to keep those $N$ for which $|b_0(N)-b_0^*|<10 \sqrt{a_1 \log a_1}$, leading to
\begin{equation}\label{hrestimate2}
h(r) = \log \sum_{\substack{c \in \mathbb{N} \\ |c-b_0^*|< 10 \sqrt{a_1 \log a_1}}} \frac{\sum_{N \in H_0(c) \cap R} P_N(r)^2}{\sum_{N \in R'}P_N(r')^2} +O(1).
\end{equation}
The map $N=\sum_{\ell =0}^{L-1}b_{\ell}(N)q_{\ell} \mapsto N':= \sum_{\ell =1}^{L-1} b_{\ell}(N) q_{\ell}'$ introduced in Section \ref{tailsection} is a bijection from $H_0(c) \cap R$ to $R'$, for those $c$ which are in the range of the summation in \eqref{hrestimate2}. To see this, note that choosing $D$ sufficiently large and using the fact that $a_1 > D$ by assumption, the condition $|c-b_0^*|< 10 \sqrt{a_1 \log a_1}$ ensures that $c>0$, and consequently $b_0 (N) >0$ whence $b_1(N) < a_2$; thus the mapping $N \mapsto N'$ is indeed well-defined. By the product form \eqref{pn/pn'productform} and Proposition \ref{tailprop}, for any $N \in H_0(c) \cap R$ we have
\[ \begin{split} \frac{P_N(r)}{P_{N'}(r')} &= \prod_{b=0}^{c-1} P_{q_0} \left( r, \frac{b q_0 \| q_0 r \| + \varepsilon_0(N)}{q_0} \right) \prod_{\ell =1}^{L-1} \exp \left( O \left( \frac{(a_2+\cdots + a_{\ell})^{3/4}}{(q_{\ell}')^{3/4}} + \frac{\log a_1}{q_{\ell}'} \right) \right) \\ &= \left( \prod_{b=0}^{c-1} f \left( (b+1) r +\varepsilon_0(N) \right) \right) \exp (O( \log a_1 ) ) . \end{split} \]
Here $(b+1)r+\varepsilon_0(N)$ is bounded away from $1$. By the definition of $R$, either $a_2 \ll a_1$ or $b_1(N) \le 0.99 a_2$, therefore Lemma \ref{epsilonlemma} gives
\[ (b+1)r + \varepsilon_0(N) \ge r+\varepsilon_0(N) \gg \frac{1}{a_1} . \]
Comparing the sum to the corresponding Riemann integral, we thus get
\[ \sum_{b=0}^{c-1} \log f \left( (b+1)r + \varepsilon_0(N) \right) = a_1 \int_0^{c/a_1} \log f(x) \, \mathrm{d} x + O(\log a_1) . \]
Here $c/a_1 = 5/6+O(\sqrt{(\log a_1)/a_1})$, hence the Taylor expansion
\[ \int_0^y \log f(x) \, \mathrm{d}x = \int_0^{5/6} \log f(x) \, \mathrm{d}x + O((y-5/6)^2) \]
and the definition \eqref{volume_4_1} of $\V$ show that
\[ \sum_{b=0}^{c-1} \log f \left( (b+1)r+\varepsilon_0(N) \right) = \frac{\V}{4 \pi} a_1 + O (\log a_1) . \]
For any $N \in H_0(c) \cap R$ we thus have
\[ \frac{P_N(r)}{P_{N'}(r')} = \exp \left( \frac{\V}{4 \pi} a_1 + O(\log a_1) \right) , \]
consequently
\[ \frac{\sum_{N \in H_0(c) \cap R} P_N(r)^2}{\sum_{N \in R'} P_N(r')^2} = \exp \left( \frac{\V}{2 \pi} a_1 + O(\log a_1) \right) . \]
Summing over all $c \in \mathbb{N}$ such that $|c-b_0^*|<10 \sqrt{a_1 \log a_1}$, estimate \eqref{hrestimate2} simplifies to
\[ h(r) = \frac{\V}{2 \pi} a_1 + O(\log a_1), \]
as claimed.\\

\noindent\textbf{Case 2.} Assume that $a_1 \le D$. We first claim that
\begin{equation}\label{b1<a2}
\sum_{N \in R} P_N(r)^2 \ll \sum_{\substack{N \in R \\ b_1(N)<a_2}} P_N(r)^2 .
\end{equation}
To any $N \in R$ such that $b_1(N)=a_2$, let us associate $N^*:=N-q_1$. Note that $N^*$ is obtained from $N$ by reducing the Ostrowski digit $b_1(N)=a_2$ by $1$, i.e.\ $b_1(N^*)=a_2-1$ and $b_{\ell}(N^*)=b_{\ell}(N)$ for all $\ell \neq 1$. We necessarily have $b_0(N)=b_0(N^*)=0$, and by definition \eqref{epsilondef}, $\varepsilon_{\ell}(N^*)=\varepsilon_{\ell}(N)$ for all $\ell \ge 1$. The product form \eqref{pnproductform} thus gives
\[ \log P_{N^*}(r) - \log P_N(r) = - \log P_{q_1} \left( r, \frac{(a_2-1)q_2 \| q_2 r \| + \varepsilon_2(N)}{q_2} \right) . \]
Following the steps in the proof of Proposition \ref{local5/6prop} (ii), here
\[ \begin{split} -\log P_{q_1} \left( r, \frac{(a_2-1)q_2 \| q_2 r \| + \varepsilon_2(N)}{q_2} \right) &\ge -C \left( 1+I_{\{ a_2=1 \}}I_{\{ b_2(N)>0.99 a_3 \}} a_3 +\log a_1 \right) \\ &\ge -C , \end{split} \]
hence $P_N(r) \ll P_{N^*}(r)$. Since the map $N \mapsto N^*$ is an injection from $\{ N \in R \, : \, b_1(N)=a_2 \}$ to $\{ N \in R \, : \, b_1(N)<a_2 \}$, the estimate \eqref{b1<a2} follows.

In particular, \eqref{hrestimate1} gives
\begin{equation}\label{hrestimate3}
h(r) = \log \sum_{c=0}^{a_1-1} \frac{\sum_{N \in H_0(c) \cap R, \,\, b_1(N)<a_2} P_N(r)^2}{\sum_{N \in R'} P_N(r')^2} + O(1) .
\end{equation}
The map $N=\sum_{\ell =0}^{L-1} b_{\ell}(N) q_{\ell} \mapsto N':= \sum_{\ell =1}^{L-1} b_{\ell}(N) q_{\ell}'$ introduced in Section \ref{tailsection} is a bijection from $\{ N \in H_0(c) \cap R \, : \, b_1(N)<a_2 \}$ to $R'$ (where the condition $b_1(N) < a_2$ ensures that the mapping $N \mapsto N'$ is indeed well-defined). The product form \eqref{pn/pn'productform} and Proposition \ref{tailprop} now yield
\[ \frac{P_N(r)}{P_{N'}(r')} = \left( \prod_{b=0}^{c-1} f \left( (b+1)r+\varepsilon_0(N) \right) \right) \exp (O(1)) . \]
By the definition of $R$, either $a_2 \ll 1$ or $b_1(N) \le 0.99 a_2$, therefore Lemma \ref{epsilonlemma} gives
\[ (b+1)r+\varepsilon_0(N) \ge r+\varepsilon_0(N) \gg 1. \]
The same points are also bounded away from $1$:
\[ (b+1)r+\varepsilon_0(N) \le (a_1-1)r+\| a_1 r \| \le 1-r, \]
therefore $1 \ll \prod_{b=0}^{c-1} f((b+1)r+\varepsilon_0(N)) \ll 1$. In particular, $1 \ll P_N(r)/P_{N'}(r') \ll 1$, and we obtain
\[ 1 \ll \frac{\sum_{N \in H_0(c) \cap R, \,\, b_1(N)<a_2} P_N(r)^2}{\sum_{N \in R'} P_N(r')^2} \ll 1 . \]
The estimate \eqref{hrestimate3} thus simplifies to $h(r)=O(1)$, as claimed.
\end{proof}

\section{Value distribution of $\mathbf{J}_{4_1}$} \label{41_section}

In this section we give the proof of Theorem \ref{th2}. We follow the strategy in \cite[Theorem 4]{BD1}. However, since we will have to work with weaker assumptions than those which are presupposed there, several modifications of the argument are necessary. 

\begin{proof}[Proof of Theorem \ref{th2}] Let $Tx=\{ 1/x \} ~(x \neq 0)$, $T0=0$ denote the Gauss map on the interval $[0,1)$. Let $r \in (0,1)$ be a rational number with continued fraction expansion $r=[0;a_1, \ldots, a_L]$. By definition, $h(x)=\log \mathbf{J}_{4_1} (x) - \log \mathbf{J}_{4_1} (Tx)$ and $\log \mathbf{J}_{4_1} (0)=0$, hence
\[ \begin{split} \log \mathbf{J}_{4_1} (r) &= \sum_{\ell=1}^L h(T^{\ell-1}r) \\ &=\sum_{\ell=1}^L \left( \frac{\textup{Vol}(4_1)}{2 \pi T^{\ell-1}r} + h(T^{\ell-1}r) - \frac{\textup{Vol}(4_1)}{2 \pi T^{\ell-1}r} \right) \\ &= \sum_{\ell=1}^L \left( \frac{\textup{Vol}(4_1)}{2 \pi}a_{\ell} + h(T^{\ell-1}r) - \frac{\textup{Vol}(4_1)}{2 \pi T^{\ell-1}r} + \frac{\textup{Vol}(4_1)}{2 \pi} T^{\ell-1}r \right) - \frac{\textup{Vol}(4_1)}{2 \pi} r, \end{split} \]
where we used that $1/T^{\ell-1}r = [a_{\ell}; a_{\ell+1}, \ldots, a_L] = a_{\ell} + T^{\ell} r$. Setting

%For $p/q = [0;a_1,\dots, a_L]$ with convergents $p_\ell/q_\ell,~1 \leq \ell \leq L$, writing $\overline{p}$ for the multiplicative inverse of $p$ modulo $q$ we have
%\begin{eqnarray*}
%\log (\mathbf{J}_{4_1} (\overline{p}/q)) & = & \sum_{\ell=1}^L h (q_{\ell-1} / q_\ell) \\
%& = & \sum_{\ell=1}^L  \left( \frac{\textup{Vol}(4_1)}{2 \pi} \frac{q_\ell}{q_{\ell-1}} + h (q_{\ell-1} / q_\ell)  - \frac{\textup{Vol}(4_1)}{2 \pi} \frac{q_\ell}{q_{\ell-1}} \right) \\
%& = & \sum_{\ell=1}^L  \left( \frac{\textup{Vol}(4_1)}{2 \pi} a_\ell + h(q_{\ell-1}/q_\ell)  + \frac{\textup{Vol}(4_1)}{2 \pi} \left(\frac{q_{\ell-1}}{q_{\ell}} -\frac{q_\ell}{q_{\ell-1}}\right) \right) - \frac{\textup{Vol}(4_1) q_{L-1}}{2 \pi q_L},
%\end{eqnarray*}
%where we used that $q_\ell / q_{\ell-1} = a_{\ell} + q_{\ell-2}/q_{\ell-1}$.

$$
\psi^*(x) := h(x) - \frac{\textup{Vol}(4_1)}{2 \pi x} +\frac{\textup{Vol}(4_1)}{2 \pi} x,
$$
this yields
$$
\log \mathbf{J}_{4_1} (r) = \frac{\textup{Vol}(4_1)}{2 \pi} \sum_{\ell=1}^L a_\ell + \sum_{\ell=1}^L \psi^*(T^{\ell-1}r) - \frac{\textup{Vol}(4_1)}{2 \pi} r.
$$
The last term is bounded. By Case (4) of \cite[Theorem 9.4]{BD3}, the limit distribution of
\begin{equation} \label{scale}
\frac{\pi \sum_{\ell=1}^L a_\ell}{6 \log N} - \frac{2 \log \log N - 2 \gamma_0 + 2 \log(6/\pi)}{\pi}
\end{equation}
(with respect to the normalized counting measure on $F_N$, as $N \to \infty$) is the standard stable distribution with stability parameter $1$ and skewness parameter $1$; here $\gamma_0$ denotes Euler's constant.

By Corollary \ref{cor_th1}, the function $h(x)$ can be extended to a function which is almost everywhere continuous. Let $\varepsilon>0$ be given, and choose $\eta = \eta(\varepsilon) >0$ ``small.'' By Theorem \ref{th3}, the function $\psi^*(x)$ is bounded on $[\eta,1]$, and thus by the Lebesgue integrability condition it is Riemann integrable on $[\eta,1]$. Consequently, there are step functions $f^-$ and $f^+$ such that 
$$
f^-(x) \leq \psi^*(x) \leq f^+(x), \qquad \eta \leq x \leq 1,
$$
and such that
\begin{equation}  \label{fint}
\int_{\eta}^1 \left(f^+(x) - f^-(x) \right) \, \mathrm{d}x \leq \varepsilon/2.
\end{equation}
By Theorem \ref{th3}, we also have
$$
\left|\psi^*(x) \right| \leq c |\log x|
$$
for some sufficiently large constant $c$. A simple approximation argument then shows that we can find functions $g^-$ and $g^+$ such that
\begin{itemize}
 \item $g^-(x) \leq \psi^*(x) \leq g^+(x)$ for $x \in (0,1)$, 
 \item $g^- (x) = c \log x$ and $g^+(x) = c |\log x|$ for $x \in (0,\eta]$,
 \item $g^-(x)$ and $g^+(x)$ are Lipschitz-continuous on $[\eta,1]$,
 \item $\int_0^1 \left( g^+(x) - g^-(x) \right) \, \mathrm{d}x \leq \varepsilon$, provided that $\eta = \eta(\varepsilon)$ was chosen sufficiently small. 
\end{itemize}
The way to obtain $g^-$ and $g^+$ is to start with $c \log x$ resp.\ $c |\log x|$ on the interval $(0,\eta]$ and with $f^-$ resp.\ $f^+$ on the interval $[\eta, 1]$, and then combine these parts into a function that is Lipschitz-continuous on $[\eta,1]$ by ``gluing together'' these functions at the discontinuities. This does not cause a problem, as there are only finitely many such discontinuities. The last property can be satisfied because of \eqref{fint}, together with the fact that $\int_0^\eta |\log x| \, \mathrm{d}x$ can be made arbitrarily small by choosing $\eta$ as small as necessary.

For an application of \cite[Corollary 3.2]{BD3} it is necessary to check that
\begin{equation} \label{bd_condition}
\sum_{n=1}^\infty \frac{1}{n^2} \left( \sup_{x \in [ \frac{1}{n+1} ,\frac{1}{n}]} |g^\pm (x) |^{\alpha_0} + \sup_{x,y \in   [\frac{1}{n+1} ,\frac{1}{n}]} \frac{|g^\pm(x) - g^\pm(y) |^{\lambda_0}}{|x-y|^{\kappa_0\lambda_0}} \right) < \infty
\end{equation}
holds true for each of the two functions $g^-$ and $g^+$ defined above, for some suitable constants $\alpha_0 > 2$ and $\lambda_0,\kappa_0 > 0$. This is indeed easily seen to be the case, since by construction the functions $g^-$ and $g^+$ are Lipschitz-continuous on $[\eta,1]$, and since they are $\pm c \log x$ near zero; we can choose for example the parameters $\alpha_0=3$, $\lambda_0=1/2$, $\kappa_0=1$ in \eqref{bd_condition}. An application of \cite[Corollary 3.2]{BD3} thus shows that for 
$$
\mu^- := \frac{12}{\pi^2} \int_0^1 \frac{g^-(x)}{1+x} \, \mathrm{d}x
$$
the limit distribution of 
$$
\frac{\sum_{\ell=1}^L  g^-(T^{\ell-1}r) - \mu^- \log N}{\sqrt{\log N}}
$$
is normal with mean zero and some variance $\sigma^2 \geq 0$; consequently with a stronger scaling factor the limit distribution of 
$$
\frac{\sum_{\ell=1}^L  g^-(T^{\ell-1}r)}{\log N}
$$
is a Dirac measure at $\mu^-$. A similar result holds for $g^+$ and $\mu^+$ in place of $g^-$ and $\mu^-$. Since $\int_0^1 \left( g^+(x) - g^-(x) \right) \, \mathrm{d}x \leq \varepsilon$, we have $\mu^+ - \mu^- \leq 2\varepsilon$. Upon letting $\varepsilon \to 0$ this implies that the limit distribution of 
$$
\frac{\log \mathbf{J}_{4_1}(r)}{\frac{3 \textup{Vol}(4_1)}{\pi^2} \log N} - \frac{2 \log \log N}{\pi} + \underbrace{\left( \frac{2 \gamma_0 - 2 \log(6/\pi)}{\pi} + \frac{4}{\textup{Vol}(4_1)} \int_0^1 \frac{\psi^*(x)}{1+x} \, \mathrm{d}x \right)}_{=:D}
$$
is the standard stable distribution with stability parameter $1$ and skewness parameter $1$, as claimed.
\end{proof}

\section*{Acknowledgements}

CA was supported by the Austrian Science Fund (FWF), projects F-5512, I-3466, I-4945, P-34763, P-35322 and Y-901. BB was supported by FWF projects F-5510 and Y-901. We want to thank the two referees of this paper for their detailed and very valuable remarks.

\end{document}